\documentclass[a4paper,USenglish]{article}

\usepackage[USenglish]{babel}
\usepackage[ansinew]{inputenc}
\usepackage{amsmath}
\usepackage{enumitem}
\usepackage{amssymb}
\usepackage{amsthm}

\usepackage{microtype}
\usepackage{xspace}
\usepackage{todonotes}
\usepackage{hyperref}
\usepackage{tikz}
\usepackage{graphicx}
\usepackage{pgfplots}
\usepackage[ruled,vlined]{algorithm2e}
\usepackage{thmtools}
\usepackage{cleveref}

\Crefname{algocfline}{Algorithm}{Algorithms}
\Crefname{algocf}{Algorithm}{Algorithms}

\pgfplotsset{compat=1.14}

\usetikzlibrary{arrows.meta,arrows}
\usetikzlibrary{patterns}
\usepgfplotslibrary{fillbetween}

\pgfplotsset{every x tick label/.append style={font=\small, yshift=0.0ex}}
\pgfplotsset{every y tick label/.append style={font=\small, yshift=0.0ex}}
\usetikzlibrary{patterns}
\makeatletter
\pgfdeclarepatternformonly[\LineSpace]{my north east lines}{\pgfqpoint{-1pt}{-1pt}}{\pgfqpoint{\LineSpace}{\LineSpace}}{\pgfqpoint{\LineSpace}{\LineSpace}}%
{
    \pgfsetcolor{\tikz@pattern@color}
    \pgfsetlinewidth{0.4pt}
    \pgfpathmoveto{\pgfqpoint{0pt}{0pt}}
    \pgfpathlineto{\pgfqpoint{\LineSpace + 0.1pt}{\LineSpace + 0.1pt}}
    \pgfusepath{stroke}
}
\makeatother
\newdimen\LineSpace
\tikzset{
    line space/.code={\LineSpace=#1},
    line space=10pt
}
\usepgfplotslibrary{fillbetween}

\tikzset{
  schraffiert/.style={pattern=horizontal lines,pattern color=#1},
  schraffiert/.default=black
}
\tikzstyle{densely dashed}=          [dash pattern=on 6pt off 2pt]

\theoremstyle{definition}
\newtheorem{Sat}{Satz}
\numberwithin{Sat}{section}
\newtheorem{theorem}[Sat]{Theorem}
\newtheorem{lemma}[Sat]{Lemma}

\newtheorem{proposition}[Sat]{Proposition}

\def\myoverset#1#2#3{\stackrel{\text{\makebox[#3pt]{#1}}}{#2}}

\newcommand{\N}{\mathbb{N}}

\newcommand{\TSP}{\textsc{TSP}\xspace}
\newcommand{\dar}{\textsc{Dial-a-Ride}\xspace}

\newcommand{\smartstart}{\textsc{Smartstart}\xspace}
\newcommand{\ignore}{\textsc{Ignore}\xspace}
\newcommand{\opt}{\textsc{Opt}\xspace}

\newcommand{\smartstartS}{\textsc{Smartstart}(\sigma)}
\newcommand{\ignoreS}{\textsc{Ignore}(\sigma)}
\newcommand{\optS}{\textsc{Opt}(\sigma)}

\newcommand{\timeReq}{r} 
\newcommand{\timeSch}{t} 
\newcommand{\posReq}{a} 
\newcommand{\posSch}{p} 

\DeclareMathOperator*{\argmin}{argmin}

\begin{document}

\title{Tight Analysis of the Smartstart Algorithm for Online Dial-a-Ride on the Line\footnote{This work was supported by the `Excellence Initiative' of the German Federal and State Governments and the Graduate School~CE at TU~Darmstadt.}}

\author{Alexander Birx \qquad\qquad Yann Disser\\ \footnotesize{Institute of Mathematics and Graduate School CE, TU Darmstadt, Germany}}

\maketitle

\begin{abstract}
The online \dar problem is a fundamental online problem in a metric space, where transportation requests appear over time and may be served in any order by a single server with unit speed. 
Restricted to the real line, online \dar captures natural problems like controlling a personal elevator. 
Tight results in terms of competitive ratios are known for the general setting and for online \TSP on the line (where source and target of each request coincide). 
In contrast, online \dar on the line has resisted tight analysis so far, even though it is a very natural online problem.

We conduct a tight competitive analysis of the \smartstart algorithm that gave the best known results for the general, metric case.
In particular, our analysis yields a new upper bound of 2.94 for open, non-preemptive online \dar on the line, which improves the previous bound of~3.41 [Krumke'00].
The best known lower bound remains 2.04 [SODA'17].
We also show that the known upper bound of 2 [STACS'00] regarding \smartstart's competitive ratio for closed, non-preemptive online \dar is tight on the line. 
\end{abstract}
\newpage

\section{Introduction}\label{section: Introduction}

Online optimization deals with settings where algorithmic decisions have to be made over time without knowledge of the future.
A typical introductory example is the problem of controlling an elevator/conveyor system, where requests to transport passengers/goods arrive over time and the elevator needs to decide online how to adapt its trajectory along the real line.
In terms of competitive analysis, the central question in this context is how much longer our trajectory will be in the worst-case, relative to an optimum offline solution that knows all requests ahead of time, i.e., we ask for solutions with good \emph{competitive ratio}.

While the elevator problem is a natural online problem, even simplified versions of it have long resisted tight analysis. 
\emph{Online \TSP on the line} is such a simplification, where a single server on the real line needs to serve requests that appear over time at arbitrary positions by visiting their location, i.e., requests do not need to be transported.
We distinguish the \emph{closed} and \emph{open} variants of this problem, depending on whether the server needs to eventually return to the origin or not.
Determining the exact competitive ratios for either variant had been an open problem for more than two decades~\cite{Ausiello1, BlomKrumkePaepeStougie/01, JailletWagner/08, Krumke1, Krumke3, Lipmann/03}, when Bjelde et al.~\cite{Disser1} were finally able to conduct a tight analysis that established competitive ratios of roughly~1.64 for the closed case and~2.04 for the open case.

The next step towards formally capturing the intuitive elevator problem is to allow transportation requests that appear over time; and to fix a capacity $c \in \mathbb{N} \cup \{\infty\}$ of the server that limits the number of transportation requests that can be served simultaneously.
The resulting \emph{online \dar problem on the line} has received considerable attention in the past~\cite{Ascheuer1, Disser1, Feuerstein1, Krumke1, Krumke2, Lipmann/03}, but still resists tight analysis.
The best known (non-preemptive) bounds put the competitive ratio in the range $[1.75,2]$ for the closed variant (see~\cite{Disser1, Ascheuer1}). For the open variant the best known (non-preemptive) bounds put the competitiv ratio in the range $[2.04,3]$ for $c=1$ and in the range $[2.04,3.41]$ for $c>1$ (see~\cite{Disser1, Krumke1}).
In this paper, we show an improved upper bound of (roughly) $2.94$ for open online \dar on the line for arbitrary capacity $c\in\N\cup\{\infty\}$.

A straight-forward algorithm for online \dar on the line is the algorithm $\ignore$~\cite{Ascheuer1}:
Whenever the server is idle and unserved requests $R_t$ are present at the current time~$t$, compute an optimum schedule to serve these requests from the current location, and follow this schedule while \emph{ignoring} newly incoming requests.
$\ignore$ has a competitive ratio of exactly~$4$ (see \Cref{appendix: Algorithm Ignore}).
This competitive ratio can be improved by potentially waiting before starting the optimum schedule, in order to protect against requests that come in right after we decide to start.
Ascheuer et al.~\cite{Ascheuer1} proposed the algorithm $\smartstart$ (see \Cref{algorithm: Smartstart}) that delays starting the optimum schedule until a certain time $t$ relative to the length $L(\timeSch,\posSch,R_t)$ of this schedule (formal definitions below).

$\smartstart$ is parameterized by a factor $\Theta > 1$ that scales this waiting period.
In this paper, we conduct a tight analysis of the best competitive ratio of \smartstart for open/closed online \dar on the line, over all parameter values $\Theta > 1$.

\subparagraph*{Results and techniques.}

The \smartstart algorithm is of particular importance for online \dar, since, on arbirary metric spaces, it achieves the best possible competitive ratio of $2$ for the closed variant~\cite{Ascheuer1, Ausiello1}, and the best known competitive ratio of~$2+\sqrt{2} \approx 3.41$ for the open variant~\cite{Krumke1}.
We provide a conclusive treatment of this algorithm for online \dar on the line in terms of competitive analysis, both for the open and the closed variant.

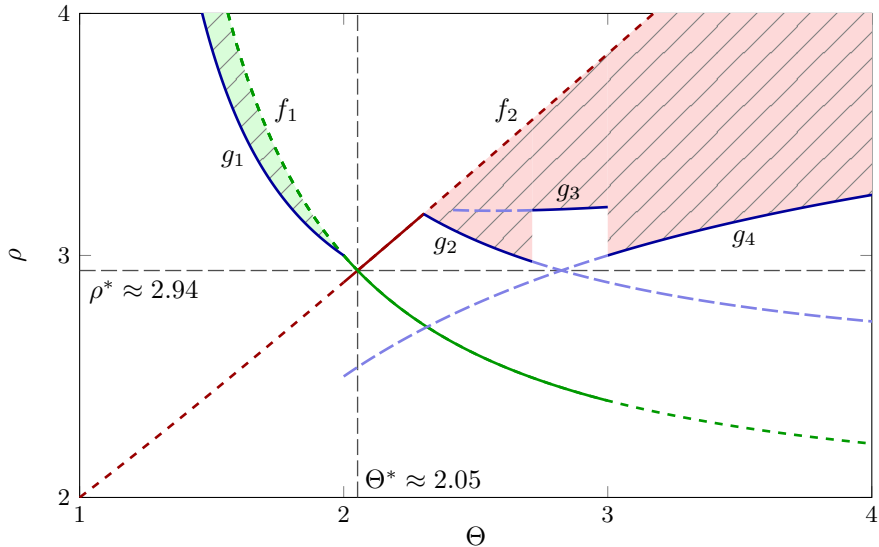
\begin{figure}[h]
\centering\begin{tikzpicture}
\begin{axis}[
samples=100,
xmin=1,
xmax=4,
xtick={1,2,...,4},
ytick={2,3,4},
ymin=2,
ymax=4,
xlabel=$\Theta$,
ylabel=$\rho$,
height=8cm,
width=12cm,
y label style={yshift=0.5em},
x label style={yshift=0.5em},
legend pos=north west,
 xticklabel style={/pgf/number format/.cd,fixed}
]

\addplot[no markers, domain=2.3027756:2.7139][name path=A1, black!6!red!15,line width=0pt]{x+1-(x-1)/(3*x+3)};
\addplot[no markers, domain=2.7139:3][name path=A2, black!6!red!15,line width=0pt]{x+1-(x-1)/(3*x+3)};
\addplot[no markers, domain=3:6][name path=A3, black!6!red!15,line width=0pt]{x+1-(x-1)/(3*x+3)};

\draw[densely dashed,line width=0.5pt,black!70] (axis cs:2.0526,\pgfkeysvalueof{/pgfplots/ymin}) -- (axis cs:2.0526,\pgfkeysvalueof{/pgfplots/ymax});
\addplot[no markers,densely dashed, domain=1:6,line width=0.5pt, black!70]{2.93768};
\addplot[no markers, thick, domain=1.2:2.0526][name path=Z1, black!40!green,line width=1pt]{5};
\addplot[no markers, thick, domain=1.2:2.0526][name path=X, dashed, black!40!green,line width=1pt]{(2*x^2+2*x)/(x^2+x-2)};
\addplot[no markers, thick, domain=2:2.0526][black!40!green,line width=1pt]{(2*x^2+2*x)/(x^2+x-2)};
\addplot[no markers, thick, domain=2.0526:3][black!40!green,line width=1pt]{(2*x^2+2*x)/(x^2+x-2)};
\addplot[no markers, thick, domain=2.0526:6][name path=Z2, black!40!green,line width=1pt]{5};
\addplot[no markers, thick, domain=2.0526:6][name path=Y,black!40!red,dashed,line width=1pt]{x+1-(x-1)/(3*x+3)};
\addplot[no markers, thick, domain=2:2.0526][black!40!red,line width=1pt]{x+1-(x-1)/(3*x+3)};
\addplot[no markers, thick, domain=2.0526:2.3027756][black!40!red,line width=1pt]{x+1-(x-1)/(3*x+3)};

\addplot[no markers, thick, domain=1.2:6][black!40!green,dashed,line width=1pt]{(2*x^2+2*x)/(x^2+x-2)};
\addplot[no markers, thick, domain=1.2:2][name path=C3, dashed,black!40!green,line width=1pt]{(2*x^2+2*x)/(x^2+x-2)};

\addplot[no markers, thick, domain=1:2][black!40!red,dashed,line width=1pt]{x+1-(x-1)/(3*x+3)};

\addplot[no markers, thick, domain=1.2:2][name path=D3, black!40!blue,line width=1pt]{(3*x^2)/(x^2+x-2)};

\addplot[no markers, thick, domain=2.3027756:2.7139][name path=B1, black!40!blue,line width=1pt]{((3*x+3)/(x-1)+2+2/x)/((2*x+3)/x)};
\addplot[no markers, densely dashed, domain=2.7139:4.64575][black!20!blue!50,line width=1pt]{((3*x+3)/(x-1)+2+2/x)/((2*x+3)/x)};

\addplot[no markers, thick, domain=2.7139:3][name path=B2, black!40!blue,line width=1pt]{(((3*x-1)/(x-1))+1+(1/x))/(1+(2/x))};

\addplot[black!6!red!15, postaction={pattern=my north east lines,pattern color=gray}] fill between[of=A1 and B1];
\addplot[black!6!red!15, postaction={pattern=my north east lines,pattern color=gray}] fill between[of=A2 and B2];
\addplot[no markers, densely dashed, domain=2.41421356237:2.7139][black!20!blue!50,line width=1pt]{((3*x-1)/(x-1)+1+1/x)/(1+2/x)};

\addplot[no markers, thick, domain=3:6][name path=B3, black!40!blue,line width=1pt]{4-3/x};
\addplot[no markers, densely dashed, domain=2:3][black!20!blue!50,line width=1pt]{4-3/x};

\addplot[black!6!red!15, postaction={pattern=my north east lines,pattern color=gray}] fill between[of=A3 and B3];
\addplot[black!6!green!15, postaction={pattern=my north east lines,pattern color=gray}] fill between[of=C3 and D3];

\node[right] at (axis cs: 1.7,3.6) {$f_1$};
\node[left] at (axis cs: 2.7,3.6) {$f_2$};
\node[right] at (axis cs: 1.5,3.4) {$g_1$};
\node[left] at (axis cs: 2.47,3.04) {$g_2$};
\node[right] at (axis cs: 2.77,3.26) {$g_3$};
\node[left] at (axis cs: 3.6,3.07) {$g_4$};
\node[above] at (axis cs: 2.3,2) {$\Theta^*\approx 2.05$};
\node[below right] at (axis cs: 1,2.93768) {$\rho^*\approx 2.94$};
\end{axis}
\end{tikzpicture}
\caption{Overview over our bounds for \smartstart. 
The functions~$f_1$ (green) / $f_2$ (red) are upper bounds for the cases where \smartstart waits / does not wait before starting the final schedule, respectively.
The upper bounds are drawn solid in the domains where they are tight for their corresponding case.
The functions $g_1$ through $g_4$ (blue) are general lower bounds; dashed continuations indicate how far these bounds could be extended.}
\label{figure: Bounds}
\end{figure}

Regarding the open case, we show that $\smartstart$ attains a competitive ratio of $\rho^* \approx 2.94$ for parameter value $\Theta^* \approx 2.05$ (\Cref{section: Upper Bound for the Open Version}).
To show this, we derive two separate upper bounds depending on $\Theta$ (cf.~\Cref{figure: Bounds}): an upper bound $f_1(\Theta)$ for the case that \smartstart has a waiting period before starting its last schedule (\Cref{proposition: Upper Bound Waiting}), and an upper bound~$f_2(\Theta)$ for the case that \smartstart begins its final schedule immediately (\Cref{proposition: Upper Bound No Waiting}).
The resulting general upper bound of $\max\{f_1(\Theta), f_2(\Theta)\}$ has its minimum precisely at the intersection point~$(\Theta^*, \rho^*)$ of~$f_1$ and~$f_2$.

On the other hand, we show that for $\Theta \in (2,3)$ there are instances where \smartstart waits before starting its final schedule and has competitive ratio at least $f_1(\Theta)$ (\Cref{proposition: Lower Bound Waiting}).
Similarly, we show that for $\Theta \in [2,2.303]$ there are instances where \smartstart does not wait before starting its final schedule and has competitive ratio at least $f_2(\Theta)$ (\Cref{proposition: Lower Bound No Waiting}).
Together, this implies that the general upper bound of $\max\{f_1(\Theta), f_2(\Theta)\}$ is tight for~$\Theta \in (2,2.303]$, and thus for~$\Theta = \Theta^*$ (cf.~\Cref{figure: Bounds}).

To complete our analysis of \smartstart, we give lower bound constructions for different domains of~$\Theta$ ($g_1$ through $g_4$ in~\Cref{figure: Bounds}) that establish that $\Theta^*$ is indeed the best parameter choice for \smartstart in the worst-case (\Cref{lemma: Remaining Lower Bounds}).
The key ingredient to all our lower bounds is a way to \emph{lure} \smartstart away from the origin (\Cref{lemma: Luring}).

Finally, for the closed variant of the problem, we provide a lower bound of~$2$ on the best-possible competitive ratio of \smartstart over all possible choices of the parameter~$\Theta>1$ (\Cref{section: Lower Bound for the Closed Version}).
This tightly matches the known upper bound for general metric spaces~\cite{Ascheuer1}.

\subparagraph*{Significance.}

The main contribution of this paper is a conclusive treatment of the algorithm \smartstart for online \dar on the line in terms of competitive analysis.
Additionally, our analysis yields an improved upper bound of (roughly)~$2.94$ for non-preemptive, open online \dar on the line. 
This is the first bound below $3$ and narrows the gap for the competitive ratio to $[2.04,2.94]$.
Our work is likely to serve as a starting point towards devising better algorithms (preemptive or non-preemptive) that narrow the gaps for both the open and closed setting by avoiding critical ``mistakes'' of \smartstart, as evidenced by our lower bound constructions

\subparagraph*{Further related work.}

In this paper, we focus on the non-preemptive variant of online \dar on the line, where requests cannot be unloaded on the way in reaction to the arrival of new requests.
For the case where preemption is allowed, the best known bounds for the closed version are $[1.64,2]$ (see~\cite{Ausiello1, Ascheuer1}), which is slightly worse than the gap of $[1.75,2]$ in the non-preemptive case. 
On the other hand, the best bounds for the open, preemptive variant are $[2.04,2.41]$ (see~\cite{Disser1}), which is better than the gap of $[2.04,2.94]$ in the non-preemptive case.
In particular, the preemptive and non-preemptive cases can currently not be separated in terms of competitive ratios.

A variant of the online \dar problem where the objective is to minimize the maximal flow time, instead of the makespan, has been studied by Krumke et al.~\cite{Krumke2, Krumke3}.
They established that in many metric spaces no online algorithm can be competitive with respect to this objective. 
Hauptmeier et al.~\cite{Hauptmeier1} showed that a competitive algorithm is possible if we restrict ourselves to instances with ``reasonable'' load, which roughly means that requests that appear over a sufficiently large time period~$T$ can always be served in time at most~$T$.

Lipmann et al.~\cite{Lipmann1} studied a natural variant of closed, online \dar where the destinations of requests are only revealed upon collection by the server.
For general metric spaces and server capacity~$c$, they showed a tight competitive ratio of~$3$ in the preemptive setting, and lower/upper bounds of~$\max\{3.12,c\}$ and $2c+2$, respectively, in the non-preemptive setting.

Yi and Tian \cite{Yi2} considered the online \dar problem with deadlines, with the objective of serving the maximum number of requests.
They provided bounds on the competitive ratio depending on the diameter of the metric space.
In~\cite{Yi1} they further studied this setting when the destination of requests are only revealed upon collection by the server.

The offline version of \dar on the line has been studied in various settings, for an overview see~\cite{Paepe1}.
For the closed, non-preemptive case without release times, Gilmore and Gomory~\cite{GilmoreGomory/64} and Atallah and Kosaraju~\cite{Atallah1} gave a polynomial time algorithm for a server with unit capacity $c=1$, and Guan~\cite{Guan1} showed that the problem is hard for~$c=2$.
Bjelde et al.~\cite{Disser1} extended this result to any finite $c \geq 2$ and both the open and closed case.
They further showed that with release times the problem is already hard for finite~$c \geq 1$.
On the other hand, the complexity of the case $c=\infty$ has not yet been established.
The closed, preemptive case without release times was shown to be polynomial time solvable for $c=1$ by Atallah and Kosaraju \cite{Atallah1}, and for $c \geq 2$ by Guan~\cite{Guan1}.

For the closed, non-preemptive case with finite capacity, Krumke~\cite{Krumke1} provided a $3$-approxi\-mation algorithm. Finally, Charikar and Raghavachari~\cite{Charikar1} gave approximation algorithms for the closed case without release times, both preemptive and non-preemptive, on general metric spaces.
They also claimed to have a 2-approximation for the line, but this result appears to be incorrect (personal communication). 

\section{Preliminaries}\label{section: Preliminaries}

Formally, an instance of \dar on the line is given by a set of requests denoted by $\sigma = \{(a_1,b_1;r_1),(a_2,b_2;r_2),\dots,(a_n,b_n;r_n)\}$ that need to be served by a single server with capacity $c \in \mathbb{N}\cup\{\infty\}$, travelling with unit speed and starting at the origin on the real line.
Request~$\sigma_i$ appears at time~$r_i > 0$ at position~$a_i \in \mathbb{R}$ of the real line and needs to be transported to position~$b_i \in \mathbb{R}$.
The objective of the \dar problem on the line is to find a shortest schedule for the server to transport all requests without carrying more than~$c$ requests at once, where the length of a schedule is the length of the resulting trajectory.
In the \emph{closed} version of the problem, the server eventually needs to return to the origin, in the \emph{open} version it does not.
In the \emph{online} \dar problem on the line, each request~$\sigma_i$ is revealed only at time~$r_i$, and $n$ is only revealed implicitly by the fact that no more requests appear.
In contrast, in the \emph{offline} problem, all requests are known ahead of time (but release times still need to be respected).

We define~$L(\timeSch,\posSch,R)$ to be the length of a shortest schedule that starts at position~$p$ at time~$t$ and serves all requests in $R \subseteq \sigma$ after they appeared (i.e., the schedule must respect release times).
Observe that, for all $0 \leq t \leq t'$, $p, p' \in \mathbb{R}$, and $R \subseteq \sigma$, we have
\begin{align}
  L(\timeSch,\posSch,R) &\ge  L(\timeSch',\posSch,R),           \label{equation: Schedule Time}\\
  L(\timeSch,\posSch,R) &\le  |\posSch - \posSch'| + L(\timeSch,\posSch',R).  \label{equation: Schedule Triangule Eq}
\end{align}
By~$x_- := \min\{0,\min_{i=1,\dots,n} \{a_i\},\min_{i=1,\dots,n} \{b_i\}\}$ we denote the leftmost and by $x_+:=\max\{0,\max_{i=1,\dots,n} \{a_i\},\max_{i=1,\dots,n} \{b_i\}\}$ the rightmost position that needs to be visited by the server.
Here and throughout, we orient the real line from left to right.
Obviously, there is an optimum trajectory that only visits points in~$[x_-, x_+]$, and we let \opt be such a trajectory and~$\opt(\sigma) := L(0,0,\sigma)$ be its length.

For the description of online algorithms, we denote by $t$ the current time and by $R_t$ the set of requests that have appeared until time~$t$ but have not been served yet.
The algorithm \smartstart is given in \Cref{algorithm: Smartstart}. 
Essentially, \smartstart waits before starting an optimal schedule to serve all available requests at time
\begin{equation}
  \min_{\timeSch' \ge \timeSch}\left\{\timeSch' \ge\frac{L(\timeSch',\posSch,R_{t'})}{\Theta-1}\right\},\label{equation: Smartstart Definition}
\end{equation}
where $\posSch$ is the current position of the server and $\Theta > 1$ is a parameter of the algorithm that scales the waiting time. Importantly, \smartstart ignores incoming requests while executing a schedule. 
Whenever we need to distinguish the behavior of \smartstart for different values of~$\Theta > 1$, we write $\smartstart_{\Theta}$ to make the choice of~$\Theta$ explicit.
The length of \smartstart's trajectory is denoted by $\smartstart(\sigma)$.
Note that the schedules used by \smartstart are NP-hard to compute for $1 < c < \infty$, see~\cite{Disser1}. 

\begin{algorithm}[H]
\SetKwBlock{Repeat}{repeat}{}
\DontPrintSemicolon
$\posSch_1\leftarrow 0$\;

\For{$j = 1,2,\dots$}{
\While{$t\le L(\timeSch,\posSch_j,R_t)/(\Theta-1)$}{wait\;}
$\timeSch_j\leftarrow t$\;
$S_j\leftarrow$ optimal offline schedule serving $R_t$ starting from $\posSch_j$\;
execute $S_j$\;
$\posSch_{j+1}\leftarrow $ current position}

 \caption{$\smartstart$}\label{algorithm: Smartstart}
\end{algorithm}

We let $N\in \mathbb{N}$ be the number of schedules needed by \smartstart to serve~$\sigma$.
The $j$-th schedule is denoted by $S_j$, its starting time by~$\timeSch_j$, its starting point by~$\posSch_j$, its ending point by~$\posSch_{j+1}$ (cf.~\Cref{algorithm: Smartstart}), and the set of requests served in~$S_j$ by~$\sigma_{S_j}$.
For convenience, we set $\timeSch_0 = \posSch_0 = 0$.
Finally, we denote by $y_-^{S_j}$ the leftmost and by $y_+^{S_j}$ the rightmost position that occurs in the requests $\sigma_{S_j}$. Note that $y_-^{S_j}$ and $y_+^{S_j}$ need not lie on different sides of the origin, in contrast to $x_{-/+}$.

\section{Upper Bound for the Open Version}\label{section: Upper Bound for the Open Version}

In this section, we give an upper bound on the completion time
\begin{equation}
\smartstartS=\timeSch_N+L(\timeSch_N,\posSch_N,\sigma_{S_N})\label{equation: Costs Smartstart}
\end{equation}
of \smartstart, relative to~$\opt(\sigma)$.
To do this, we consider two cases, depending on whether or not \smartstart waits after finishing schedule~$S_{N-1}$ and before starting the final schedule~$S_N$.
If \smartstart waits, the starting time of schedule~$S_N$ is given by
\begin{equation}
\timeSch_N=\frac{1}{\Theta-1}L(\timeSch_N,\posSch_N,\sigma_{S_N}),\label{equation: Starting Time No Wait}
\end{equation}
otherwise, we have
\begin{equation}\label{equation: Starting Time Wait}
\timeSch_N=\timeSch_{N-1}+L(\timeSch_{N-1},\posSch_{N-1},\sigma_{S_{N-1}}).
\end{equation}
We start by giving a lower bound on the starting time of a schedule.

\begin{lemma}\label{lemma: Lower Bound Starting Time}
Algorithm $\smartstart$ does not start schedule $S_j$ earlier than time $\frac{|\posSch_{j+1}|}{\Theta}$, i.e., we have $\timeSch_j\ge\frac{|\posSch_{j+1}|}{\Theta}$.
\end{lemma}
\begin{proof}
Since $\smartstart$ at least has to move from $\posSch_j$ to $\posSch_{j+1}$, we have
\begin{equation*}
L(\timeSch_j,\posSch_j,\sigma_{S_j})\ge |\posSch_j-\posSch_{j+1}|.
\end{equation*}
Note however that $\smartstart$ needs at least time $|\posSch_j|$ to reach $\posSch_j$. Therefore, we have
\begin{align}
\timeSch_j&\ge\min\{t\ge |\posSch_j|: t+|\posSch_j-\posSch_{j+1}|\le \Theta t\}\nonumber\\
&=\min\left\{t\ge |\posSch_j|: \frac{|\posSch_j-\posSch_{j+1}|}{\Theta-1}\le t\right\}\nonumber\\
&=\max\left\{|\posSch_j|,\frac{|\posSch_j-\posSch_{j+1}|}{\Theta-1}\right\}.\label{equation: Starting Time Trivial Approx}
\end{align}
It remains to show
\begin{equation*}
\max\left\{|\posSch_j|,\frac{|\posSch_j-\posSch_{j+1}|}{\Theta-1}\right\}\ge\frac{|\posSch_{j+1}|}{\Theta}.
\end{equation*}
For $|\posSch_j|\ge \frac{|\posSch_{j+1}|}{\Theta}$ we trivially have
\begin{equation}
\max\left\{|\posSch_j|,\frac{|\posSch_j-\posSch_{j+1}|}{\Theta-1}\right\}\ge |\posSch_j|\ge\frac{|\posSch_{j+1}|}{\Theta}.\label{equation: Starting Time First Approx}
\end{equation}
For $|\posSch_j|< \frac{|\posSch_{j+1}|}{\Theta}$ we have
\begin{equation}
|\posSch_j-\posSch_{j+1}|\overset{|\posSch_j|< \frac{|\posSch_{j+1}|}{\Theta}< |\posSch_{j+1}|}{>}\left|\frac{\posSch_{j+1}}{\Theta}-\posSch_{j+1}\right|.\label{equation: Starting Time Second Approx}
\end{equation}
This leads to
\begin{align}
\max\left\{|\posSch_j|,\frac{|\posSch_j-\posSch_{j+1}|}{\Theta-1}\right\}&\myoverset{}{\ge}{15} \frac{|\posSch_j-\posSch_{j+1}|}{\Theta-1}\nonumber\\
&\myoverset{(\ref{equation: Starting Time Second Approx})}{>}{15} \frac{|\frac{\posSch_{j+1}}{\Theta}-\posSch_{j+1}|}{\Theta-1}\nonumber\\
&\myoverset{}{=}{15}\frac{|\posSch_{j+1}|}{\Theta}\label{equation: Starting Time Third Approx}
\end{align}
To sum it up, we have
\begin{equation*}
\timeSch_j\overset{(\ref{equation: Starting Time Trivial Approx})}{\ge}\max\left\{|\posSch_j|,\frac{|\posSch_j-\posSch_{j+1}|}{\Theta-1}\right\}\overset{(\ref{equation: Starting Time Third Approx}),(\ref{equation: Starting Time First Approx})}{\ge}\frac{|\posSch_{j+1}|}{\Theta},
\end{equation*}
as claimed.
\end{proof}

The following bound on the length of \smartstart's schedules is an essential ingredient in our upper bounds.

\begin{lemma}\label{lemma: Costs per Schedule}
For every schedule $S_j$ of $\smartstart$, we have
\begin{equation*}
L(\timeSch_j,\posSch_j,\sigma_{S_j})\le \left(1+\frac{\Theta}{\Theta+2}\right)\optS.  
\end{equation*}
\end{lemma}
\begin{proof}
First, we notice that by the triangle inequality we have
\begin{equation}\label{equation: Triangle Zero}
L(\timeSch_j,\posSch_j,\sigma_{S_j})\le |\posSch_j|+ L(\timeSch_j,0,\sigma_{S_j})\le \optS+|\posSch_j|.
\end{equation}
Now, let~$\sigma^{\opt}_{S_j}$ be the first request of~$\sigma_{S_j}$ that is picked up by $\opt$ and let~$\posReq^\opt_j$ be its starting point and~$\timeReq^\opt_j$ be its release time. We have
\begin{equation}\label{equation: Triangle First Opt}
L(\timeSch_j,\posSch_j,\sigma_{S_j})\le |\posReq^\opt_j-\posSch_j|+ L(\timeSch_j,\posReq^\opt_j,\sigma_{S_j}),
\end{equation}
again by the triangle inequality. Since \opt serves all requests of $\sigma_{S_j}$ starting at position $\posReq^\opt_j$ no earlier than time $\timeReq^\opt_j$, we have
\begin{equation}\label{equation: Opt Final Schedule}
L(\timeSch_j,\posReq^\opt_j,\sigma_{S_j})\overset{\timeReq^\opt_j\le \timeSch_j}{\le}L(\timeReq^\opt_j,\posReq^\opt_j,\sigma_{S_j})\le \optS-\timeReq^\opt_j,
\end{equation}
which yields
\begin{align}
L(\timeSch_j,\posSch_j,\sigma_{S_j})&\myoverset{(\ref{equation: Triangle First Opt})}{\le}{35}|\posReq^\opt_j-\posSch_j|+ L(\timeSch_j,\posReq^\opt_j,\sigma_{S_j})\nonumber\\
&\myoverset{(\ref{equation: Opt Final Schedule})}{\le}{35}\optS+|\posReq^\opt_j-\posSch_j|-\timeReq^\opt_j\nonumber\\
&\myoverset{$\timeSch_{j-1}<\timeReq^\opt_j$}{<}{35}\optS+|\posReq^\opt_j-\posSch_j|-\timeSch_{j-1}.\label{equation: Schedule Length Via Opt Start}
\end{align}
Since $\posSch_j$ is the destination of a request, \opt needs to visit it. In the case that $\opt$ visits $\posSch_j$ before collecting $\sigma^\opt_{S_j}$, \opt still has to collect and serve every request of $\sigma_{S_j}$ after it has visited position $\posSch_j$ the first time, which directly implies
\begin{equation*}
\left(1+\frac{\Theta}{\Theta+2}\right)\optS>\optS\ge L(|\posSch_j|,\posSch_j,\sigma_{S_j})\overset{|\posSch_j|\le \timeSch_j}{\ge} L(\timeSch_j,\posSch_j,\sigma_{S_j}).
\end{equation*}
On the other hand, if \opt collects $\sigma^\opt_{S_j}$ before visiting the position $\posSch_j$, we have
\begin{equation}
\timeSch_{j-1}+|\posReq^\opt_j-\posSch_j|\overset{\timeSch_{j-1}<\timeReq^\opt_j}{<}\timeReq^\opt_j+|\posReq^\opt_j-\posSch_j|\le \optS,\label{equation: Opt Collect Before Visit}
\end{equation}
since \opt cannot collect $\sigma^\opt_{S_j}$ before time $\timeReq^\opt_j$ and then still has to visit position $\posSch_j$. 
Thus, we have
\begin{align}
L(\timeSch_j,\posSch_j,\sigma_{S_j})&\myoverset{(\ref{equation: Schedule Length Via Opt Start})}{<}{35} \optS+|\posReq^\opt_j-\posSch_j|-\timeSch_{j-1}\nonumber\\
&\myoverset{(\ref{equation: Opt Collect Before Visit})}{\le}{35}2\optS-2\timeSch_{j-1}\nonumber\\
&\myoverset{\text{Lem }\ref{lemma: Lower Bound Starting Time}}{\le}{35}2\optS-2\frac{|\posSch_j|}{\Theta}.\label{equation: Second Approximation For Schedule Length}
\end{align}
This implies
\begin{align*}
L(\timeSch_j,\posSch_j,\sigma_{S_j})&\myoverset{(\ref{equation: Triangle Zero}),(\ref{equation: Second Approximation For Schedule Length})}{\le}{25} \min\left\{\optS+|\posSch_j|,2\optS-\frac{2}{\Theta}|\posSch_j|\right\}\\
&\myoverset{}{\le}{25}\left(1+\frac{\Theta}{\Theta+2}\right)\optS,
\end{align*}
since the minimum above is largest for $|\posSch_j|=\frac{\Theta}{\Theta+2}\optS$.
\end{proof}

The following proposition uses \Cref{lemma: Costs per Schedule} to provide an upper bound for the competitive ratio of $\smartstart$, in the case, where $\smartstart$ does have a waiting period before starting the final schedule.

\begin{proposition}\label{proposition: Upper Bound Waiting}
In the case that $\smartstart$ waits before executing $S_N$, we have
\begin{equation*}
\frac{\smartstart(\sigma)}{\opt(\sigma)} \le f_1(\Theta) := \frac{2\Theta^2+2\Theta}{\Theta^2+\Theta-2}.
\end{equation*}
\end{proposition}
\begin{proof}
Assume $\smartstart$ waits before starting the final schedule. Then we have
\begin{equation}
\timeSch_N+L(\timeSch_N,\posSch_N,\sigma_{S_N})=\Theta \timeSch_N\label{equation: Smartstart Definition Rearranged}
\end{equation}
by definition of $\smartstart$. This implies
\begin{align*}
\smartstartS &\myoverset{(\ref{equation: Costs Smartstart})}{=}{15}\timeSch_N + L(\timeSch_N,\posSch_N,\sigma_{S_N})
\myoverset{(\ref{equation: Smartstart Definition Rearranged})}{=}{15}\Theta \timeSch_N
\myoverset{(\ref{equation: Starting Time No Wait})}{=}{15} \frac{\Theta}{\Theta-1}L(\timeSch_N,\posSch_N,\sigma_{S_N}).
\end{align*}
\Cref{lemma: Costs per Schedule} thus yields the claimed bound:
\begin{align*}
\smartstartS &\myoverset{}{=}{35} \frac{\Theta}{\Theta-1}L(\timeSch_N,\posSch_N,\sigma_{S_N})\\
&\myoverset{Lem \ref{lemma: Costs per Schedule}}{\le}{35} \frac{\Theta}{\Theta-1}\left(1+\frac{\Theta}{\Theta+2}\right)\optS\\
&\myoverset{}{=}{35} \frac{2\Theta^2+2\Theta}{\Theta^2+\Theta-2}\optS.\qedhere
\end{align*}
\end{proof}

It remains to examine the case, where the algorithm $\smartstart$ has no waiting period before starting the final schedule. We start with two lemmas that give us an upper bounds for the length of a schedule depending on its extreme positions. 

\begin{lemma}\label{lemma: Approx Schedule From Zero}
Let $S_j$ with $j\in\{1,\dots,N\}$ be a schedule of $\smartstart$. Moreover, let $\optS=|x_-|+x_++y$ for some $y\ge 0$. Then, we have
\begin{equation*}
L(\timeSch_j,0,\sigma_{S_j})\le |\min\{0,y^{S_j}_-\}|+\max\{0,y^{S_j}_+\}+y.
\end{equation*}
\end{lemma}
\begin{proof}
We need to analyze the amount of time the server needs to serve $\sigma_{S_j}$ starting from position $0$ at time $\timeSch_j$. First of all, note that the server does not wait at any point, since all requests of $\sigma_{S_j}$ already have appeared at time $\timeSch_j$. Because of that, the server cannot go to the left of $\min\{0,y^{S_j}_-\}$ or to the right of $\max\{0,y^{S_j}_+\}$ while staying on an optimal route. Furthermore, we notice that the route $\opt$ takes to serve $\sigma$ is a valid route to serve $\sigma_{S_j}$, since $\sigma_{S_j}\subseteq\sigma$. However, we can skip every part of the route $\opt$ takes to collect $\sigma$ that lies left of $\min\{0,y^{S_j}_-\}$ or right of $\max\{0,y^{S_j}_+\}$, since no requests of $\sigma_{S_j}$ have a starting or ending point that lies in those regions. Since all requests already have appeared at time $\timeSch_j$, this does not produce additional waiting time, i.e., we can just delete the parts of the route that lie left of $\min\{0,y^{S_j}_-\}$ and right of $\max\{0,y^{S_j}_+\}$ and still have a valid route for serving $\sigma_{S_j}$ when starting at time $\timeSch_j$. This shortens the length of the route by at least 
\begin{equation*}
|x_-|-|\min\{0,y^{S_j}_-\}|+x_+-\max\{0,y^{S_j}_+\},
\end{equation*}
which gives us
\begin{align*}
L(\timeSch_j,0,\sigma_{S_j})&\le\optS-(|x_-|-|\min\{0,y^{S_j}_-\}|+x_+-\max\{0,y^{S_j}_+\})\\
&=|\min\{0,y^{S_j}_-\}|+\max\{0,y^{S_j}_+\}+y,
\end{align*}
as desired.
\end{proof}

\begin{lemma}\label{lemma: Approx Schedule Only One Side}
Let $S_j$ with $j\in\{1,\dots,N\}$ be a schedule of $\smartstart$. Moreover, let $\optS=|x_-|+x_++y$ for some $y\ge 0$. Then we have
\begin{equation*}
L(\timeSch_j,\max\{0,y^{S_j}_-\}+\min\{0,y^{S_j}_+\},\sigma_{S_j})\le y^{S_j}_+-y^{S_j}_-+y.
\end{equation*}
\end{lemma}
\begin{proof}
First note that the case $\max\{0,y^{S_j}_-\}=\min\{0,y^{S_j}_+\}=0$ follows from \Cref{lemma: Approx Schedule From Zero}. Assume we have $\max\{0,y^{S_j}_-\}=y^{S_j}_-$. Then all requests of $\sigma_{S_j}$ have starting and ending point on the right side of the origin and we have
\begin{equation*}
0\le y^{S_j}_-\le y^{S_j}_+,
\end{equation*}
i.e., $\min\{0,y^{S_j}_+\}=0$. Therefore, we need to examine $L(\timeSch_j,y^{S_j}_-,\sigma_{S_j})$, i.e., the length of the optimal offline schedule serving the set of requests  $\sigma_{S_j}$ and starting from position $y^{S_j}_-$ at time $\timeSch_j$. We note that the server does not wait at any point, since all requests of $\sigma_{S_j}$ already have appeared at time $\timeSch_j$. Because of that, the server cannot go to the left of $y^{S_j}_-$ or to the right of $y^{S_j}_+$ while staying on an optimal route. Furthermore, we notice that $\opt$ cannot collect any requests of~$\sigma_{S_j}$ before passing $y^{S_j}_-$ for the first time, since $\opt$ starts at the origin. Therefore, removing the parts of the path that $\opt$ takes until it first crosses $y^{S_j}_-$, gives us a valid route to serve $\sigma_{S_j}$ since $\sigma_{S_j}\subseteq\sigma$. Additionally, we can skip every part of the route $\opt$ takes to collect requests that lie left of $0$ or right of $y^{S_j}_+$ since no requests of $\sigma_{S_j}$ have a starting or ending point that lies in those regions. Again, this does not produce additional waiting time. This shortens the length of the route by at least 
\begin{equation*}
|x_-|+y^{S_j}_-+x_+-y^{S_j}_+,
\end{equation*}
which gives us
\begin{align*}
L(\timeSch_j,y^{S_j}_-,\sigma_{S_j})&\le\optS-(|x_-|+y^{S_j}_-+x_+-y^{S_j}_+)\\
&=y^{S_j}_+-y^{S_j}_-+y,
\end{align*}
as desired. 

It remains to examine the case $\min\{0,y^{S_j}_+\}=y^{S_j}_+$. In this case all requests of $\sigma_{S_j}$ have starting and ending point on the left side of the origin and we have
\begin{equation*}
y^{S_j}_-\le y^{S_j}_+\le 0,
\end{equation*}
i.e., $\max\{0,y^{S_j}_-\}=0$. Therefore, we have to examine $L(\timeSch_j,y^{S_j}_+,\sigma_{S_j})$. From this point the proof works analogously to the former case with the roles of $y^{S_j}_+$ and $y^{S_j}_-$ switched, in particular, we have
\begin{align*}
L(\timeSch_j,y^{S_j}_+,\sigma_{S_j})&\le\optS-(|x_-|+y^{S_j}_-+x_+-y^{S_j}_+)\\
&=y^{S_j}_+-y^{S_j}_-+y,
\end{align*}
as desired.
\end{proof}
Next, we give an upper bound for the rightmost position that can be reached during a schedule.

\begin{lemma}\label{lemma: Rightmost Position}
Let $S_j$ with $j\in\{1,\dots,N\}$ be a schedule of $\smartstart$. Moreover, let $|x_-|\le x_+$ and $\optS=|x_-|+x_++y$ for some $y\ge 0$. Then, for every point~$p$ that is visited by $S_j$ we have
\begin{equation*}
p\le |\posSch_j|+|\posSch_j-\posSch_{j+1}|+y-|\min\{0,y^{S_j}_-\}|.
\end{equation*}
\end{lemma}
\begin{proof}
First of all, we notice that $S_j$ does not wait at any point since all requests of $\sigma_{S_j}$ already have appeared at time $\timeSch_j$. Because of that, $S_j$ cannot go to the left of $\min\{\posSch_j,y^{S_j}_-\}$ or to the right of $\max\{\posSch_j,y^{S_j}_+\}$ while staying on a optimal route. It suffices to show
\begin{equation}
\max\{\posSch_j,y^{S_j}_+\}\le |\posSch_j|+|\posSch_j-\posSch_{j+1}|+y-|\min\{0,y^{S_j}_-\}|.\label{equation: Characterization Rightmost Point}
\end{equation}
Since $|x_-|\le x_+$ implies $\optS\ge 2|x_-|+x_+$, we have $y\ge |x_-|\ge |\min\{0,y^{S_j}_-\}|$. Note that $|x_-|\ge |\min\{0,y^{S_j}_-\}|$ holds since we always have $|x_-|\ge 0$ and $x_-\le y^{S_j}_-$, i.e., $|x_-|\ge |y^{S_j}_-|$ if $x_-<0$ and $y^{S_j}_-<0$ holds. This implies that if we have $\max\{\posSch_j,y^{S_j}_+\}=\posSch_j$, inequality (\ref{equation: Characterization Rightmost Point}) holds.
Thus, we may assume $\max\{\posSch_j,y^{S_j}_+\}=y^{S_j}_+$ in the following. Similarly as before, if we have $y^{S_j}_+<0$, the inequality (\ref{equation: Characterization Rightmost Point}) again holds, since the right hand side is always non-negative. We may thus assume $y_+^{S_j}\ge 0$, i.e.,
\begin{equation}
\max\{0,y^{S_j}_+\}=y_+^{S_j}\label{equation: Positivity Of Y}
\end{equation}
in the following. According to the triangle inequality and \Cref{lemma: Approx Schedule From Zero}, we have
\begin{align}
L(\timeSch_j,\posSch_j,\sigma_{S_j})&\myoverset{}{\le}{35} |\posSch_j|+L(\timeSch_j,0,\sigma_{S_j})\nonumber\\
&\myoverset{Lem \ref{lemma: Approx Schedule From Zero}}{\le}{35} |\posSch_j|+|\min\{0,y^{S_j}_-\}|+\max\{0,y^{S_j}_+\}+y.\nonumber\\
&\myoverset{(\ref{equation: Positivity Of Y})}{=}{35} |\posSch_j|+|\min\{0,y^{S_j}_-\}|+y^{S_j}_++y.\label{equation: First Approx For Schedule Length}
\end{align}
For the sake of contradiction, we assume 
\begin{equation}
y^{S_j}_+> |\posSch_j|+|\posSch_j-\posSch_{j-1}|+y-|\min\{0,y^{S_j}_-\}|.\label{equation: Contradiction For Schedule Length}
\end{equation}
Since $S_j$ has to visit both extreme points $\max\{\posSch_j,y^{S_j}_+\}=y^{S_j}_+$ and $\min\{\posSch_j,y^{S_j}_-\}$, we have two possible scenarios: $S_j$ either visits $\smash{\min\{\posSch_j,y^{S_j}_-\}}$ before $\smash{y^{S_j}_+}$ or $S_j$ visits $\smash{\min\{\posSch_j,y^{S_j}_-\}}$ after $\smash{y^{S_j}_+}$. In both cases $S_j$ ends in $\posSch_{j+1}$. In the first case, we have
\begin{align}
L(\timeSch_j,\posSch_j,\sigma_{S_j})&\myoverset{}{\ge}{8} |\posSch_j-\min\{\posSch_j,y^{S_j}_-\}|+|\min\{\posSch_j,y^{S_j}_-\}-y_+^{S_j}|+|y_+^{S_j}-\posSch_{j+1}|\nonumber\\
&\myoverset{}{=}{8} \posSch_j-\min\{\posSch_j,y^{S_j}_-\}+y_+^{S_j}-\min\{\posSch_j,y^{S_j}_-\}+y_+^{S_j}-\posSch_{j+1}\nonumber\\
&\myoverset{}{=}{8} \posSch_j-2\min\{\posSch_j,y^{S_j}_-\}+2y_+^{S_j}-\posSch_{j+1}\nonumber\\
&\myoverset{(\ref{equation: Contradiction For Schedule Length})}{>}{8} \posSch_j-2\min\{\posSch_j,y^{S_j}_-\}+y_+^{S_j}+|\posSch_j|+|\posSch_j-\posSch_{j+1}|\nonumber\\
&\qquad+y-|\min\{0,y^{S_j}_-\}|-\posSch_{j+1}\nonumber\\
&\myoverset{}{\ge}{8} y_+^{S_j}+|\posSch_j|+y-|\min\{0,y^{S_j}_-\}|-2\min\{\posSch_j,y^{S_j}_-\}.\label{equation: Schedule Length Approx Two Case One}
\end{align}
In the second case, we obtain the same result
\begin{align}
L(\timeSch_j,\posSch_j,\sigma_{S_j})&\myoverset{}{\ge}{8} |\posSch_j-y_+^{S_j}|+|y_+^{S_j}-\min\{\posSch_j,y_-^{S_j}\}|+|\min\{\posSch_j,y_-^{S_j}\}-\posSch_{j+1}|\nonumber\\
&\myoverset{}{=}{8} y_+^{S_j}-\posSch_j+y_+^{S_j}-\min\{\posSch_j,y_-^{S_j}\}+\posSch_{j+1}-\min\{\posSch_j,y_-^{S_j}\}\nonumber\\
&\myoverset{}{=}{8} \posSch_{j+1}+2y_+^{S_j}-2\min\{\posSch_j,y_-^{S_j}\}-\posSch_j\nonumber\\
&\myoverset{(\ref{equation: Contradiction For Schedule Length})}{>}{8} \posSch_{j+1}+y_+^{S_j}+|\posSch_j|+|\posSch_j-\posSch_{j+1}|+y\nonumber\\
&\qquad-|\min\{0,y_-^{S_j}\}|-2\min\{\posSch_j,y_-^{S_j}\}-\posSch_j\nonumber\\
&\myoverset{}{\ge}{8} y_+^{S_j}+|\posSch_j|+y-|\min\{0,y^{S_j}_-\}|-2\min\{\posSch_j,y^{S_j}_-\}.\label{equation: Schedule Length Approx Two Case Two}
\end{align}
Now we again consider two cases.

\paragraph*{\textit{Case 1: }$\min\{\posSch_j,y^{S_j}_-\}\le 0$\\}
In this case, we claim that
\begin{equation}
-\min\{\posSch_j,y^{S_j}_-\}\ge |\min\{0,y^{S_j}_-\}|\label{equation: Absolute Value Of Minimum}
\end{equation}
holds. This is clear for $\min\{0,y_-^{S_j}\}=0$ and for $\min\{\posSch_j,y_-^{S_j}\}=y_-^{S_j}$. In the remaining case, we have $\min\{0,y_-^{S_j}\}=y_-^{S_j}$ and $\min\{\posSch_j,y_-^{S_j}\}=\posSch_j$, i.e., $\posSch_j\le y_-^{S_j}\le 0$, which implies $-\posSch_j\ge -y_-^{S_j}=|y_-^{S_j}|$ as desired. This gives us
\begin{align*}
L(\timeSch_j,\posSch_j,\sigma_{S_j})&\myoverset{(\ref{equation: Schedule Length Approx Two Case One}),(\ref{equation: Schedule Length Approx Two Case Two})}{>}{35} y_+^{S_j}+|\posSch_j|+y-|\min\{0,y^{S_j}_-\}|-2\min\{\posSch_j,y^{S_j}_-\}\\
&\myoverset{(\ref{equation: Absolute Value Of Minimum})}{\ge}{35} y_+^{S_j}+|\posSch_j|+y+|\min\{0,y^{S_j}_-\}|,
\end{align*}
which is a contradiction to inequality (\ref{equation: First Approx For Schedule Length}).

\paragraph*{\textit{Case 2: }$\min\{\posSch_j,y^{S_j}_-\}> 0$\\}
The inequality $y^{S_j}_+\ge y^{S_j}_->0$ implies 
\begin{equation}
\max\{0,y^{S_j}_-\}+\min\{0,y^{S_j}_+\}=y^{S_j}_-.\label{equation: Sum Of Min And Max}
\end{equation}
Therefore, we can apply \Cref{lemma: Approx Schedule Only One Side} and the triangle inequality to obtain
\begin{align}
L(\timeSch_j,\posSch_j,\sigma_{S_j})&\myoverset{}{\le}{50} |\posSch_j-y^{S_j}_-|+L(\timeSch_j,y^{S_j}_-,\sigma_{S_j})\nonumber\\
&\myoverset{(\ref{equation: Sum Of Min And Max}), Lem \ref{lemma: Approx Schedule Only One Side}}{\le}{50} |\posSch_j-y^{S_j}_-|+y^{S_j}_+-y^{S_j}_-+y\nonumber\\
&\myoverset{}{=}{50} \max\{\posSch_j,y^{S_j}_-\}-\min\{\posSch_j,y^{S_j}_-\}+y^{S_j}_+-y^{S_j}_-+y.\label{equation: Schedule Length Third Approx}
\end{align}
We have
\begin{equation}
\max\{\posSch_j,y^{S_j}_-\}-\min\{\posSch_j,y^{S_j}_-\}-y^{S_j}_-= \posSch_j-2\min\{\posSch_j,y^{S_j}_-\}.\label{equation: Auxiliary For Schedule Length}
\end{equation}
This gives us
\begin{align}
L(\timeSch_j,\posSch_j,\sigma_{S_j})&\myoverset{(\ref{equation: Schedule Length Third Approx})}{\le}{15} \max\{\posSch_j,y^{S_j}_-\}-\min\{\posSch_j,y^{S_j}_-\}+y^{S_j}_+-y^{S_j}_-+y\nonumber\\
&\myoverset{(\ref{equation: Auxiliary For Schedule Length})}{=}{15} \posSch_j-2\min\{\posSch_j,y^{S_j}_-\}+y^{S_j}_++y.\label{equation: Schedule Length Final Approx}
\end{align}
Finally, we have
\begin{align*}
L(\timeSch_j,\posSch_j,\sigma_{S_j})&\myoverset{(\ref{equation: Schedule Length Approx Two Case Two})}{>}{15} y_+^{S_j}+|\posSch_j|+y-|\min\{0,y^{S_j}_-\}|-2\min\{\posSch_j,y^{S_j}_-\}\\
&\myoverset{}{=}{15} y_+^{S_j}+\posSch_j+y-2\min\{\posSch_j,y^{S_j}_-\},
\end{align*}
which is a contradiction to inequality (\ref{equation: Schedule Length Final Approx}).
We conclude that (\ref{equation: Contradiction For Schedule Length}) does not hold, which in turn proves (\ref{equation: Characterization Rightmost Point}) in the case that $\max\{\posSch_j,y^{S_j}_+\}=y^{S_j}_+$ holds.
\end{proof}
Now we can give an upper bound for the competitive ratio of \smartstart if the server is not waiting before starting the final schedule.

\begin{proposition}\label{proposition: Upper Bound No Waiting}
If $\smartstart$ does not wait before executing~$S_{N}$, we have
\begin{equation*}
\frac{\smartstart(\sigma)}{\opt(\sigma)} \le f_2(\Theta) := \left(\Theta+1-\frac{\Theta-1}{3\Theta+3}\right).
\end{equation*}
\end{proposition}
\begin{proof}
Assume algorithm $\smartstart$ does not have a waiting period before the last schedule, i.e., $\smartstart$ starts the final schedule $S_N$ immediately after finishing $S_{N-1}$. Without loss of generality, we assume $|x_-|\le x_+$ throughout the entire proof by symmetry.

First of all, we notice that we may assume that $\smartstart$ executes at least two schedules in this case. Otherwise either the only schedule has length $0$, which would imply~$\optS=\smartstartS=0$, or the only schedule would have a positive length, implying a waiting period. Let~$\sigma^{\opt}_{S_N}$ be the first request of~$\sigma_{S_N}$ that is served by $\opt$ and let~$\posReq^\opt_N$ be its starting point and~$\timeReq^\opt_N$ be its release time. We have
\begin{align}
\smartstartS &\myoverset{(\ref{equation: Costs Smartstart})}{=}{45}\timeSch_N+L(\timeSch_N,\posSch_N,\sigma_{S_N})\nonumber\\
&\myoverset{(\ref{equation: Starting Time Wait})}{=}{45}\timeSch_{N-1}+L(\timeSch_{N-1},\posSch_{N-1},\sigma_{S_{N-1}})+L(\timeSch_N,\posSch_N,\sigma_{S_N})\nonumber\\
&\myoverset{$\timeSch_N\ge \timeReq^\opt_N$}{\le}{45} \timeSch_{N-1}+L(\timeSch_{N-1},\posSch_{N-1},\sigma_{S_{N-1}})+L(\timeReq^\opt_N,\posSch_N,\sigma_{S_N}).\label{equation: First Approx Upper Bound No Wait}
\end{align}
Since $\opt$ serves all requests of $\sigma_{S_N}$ after time $\timeReq^\opt_N$, starting with a request with starting point~$\posReq^\opt_N$, we also have
\begin{equation}
\optS\ge \timeReq^\opt_N+L(\timeReq^\opt_N,\posReq^\opt_N,\sigma_{S_N}).\label{equation: First Approx Opt No Wait}
\end{equation}
Furthermore, we have
\begin{equation}
\timeReq^\opt_N>\timeSch_{N-1}\label{equation: Time First Request Last Schedule Opt}
\end{equation}
since otherwise $\sigma^{\opt}_{S_N}\in\sigma_{S_{N-1}}$ would hold. This gives us
\begin{align}
\smartstartS &\myoverset{(\ref{equation: First Approx Upper Bound No Wait})}{\le}{15}\timeSch_{N-1}+L(\timeSch_{N-1},\posSch_{N-1},\sigma_{S_{N-1}})+L(\timeReq^\opt_N,\posSch_N,\sigma_{S_N})\nonumber\\
&\myoverset{(\ref{equation: Schedule Triangule Eq})}{\le}{15}\timeSch_{N-1}+L(\timeSch_{N-1},\posSch_{N-1},\sigma_{S_{N-1}})\nonumber\\
&\qquad +|\posReq^\opt_N-\posSch_N|+L(\timeReq^\opt_N,\posReq^\opt_N,\sigma_{S_N})\nonumber\\
&\myoverset{(\ref{equation: First Approx Opt No Wait})}{\le}{15}\timeSch_{N-1}+L(\timeSch_{N-1},\posSch_{N-1},\sigma_{S_{N-1}})\nonumber\\
&\qquad +|\posReq^\opt_N-\posSch_N|+\optS-\timeReq^\opt_N\nonumber\\
&\myoverset{(\ref{equation: Time First Request Last Schedule Opt})}{<}{15}L(\timeSch_{N-1},\posSch_{N-1},\sigma_{S_{N-1}})+|\posReq^\opt_N-\posSch_N|+\optS.\label{equation: Second Approx Upper Bound No Wait}
\end{align}
We denote by $\sigma^{\textsc{Smart}}_{S_{N-1}}$ the last request that is delivered during schedule $S_{N-1}$ by $\smartstart$. Note that the destination of $\sigma^{\textsc{Smart}}_{S_{N-1}}$ is $\posSch_N$. We consider two cases.

\paragraph*{\textit{Case 1: }${\normalfont\opt}$\textit{ collects }$\sigma^{\normalfont\opt}_{S_N}$\textit{ before delivering } $\sigma^{\normalfont\textsc{Smart}}_{S_{N-1}}$\\}
Obviously $\opt$ cannot collect the request $\sigma^{\opt}_{S_N}$ before its release time $\timeReq^\opt_N$. Furthermore, since $\opt$ still has to go to position $\posSch_N$ for delivering request $\sigma^{\textsc{Smart}}_{S_{N-1}}$ after collecting $\sigma^{\opt}_{S_N}$, we have
\begin{equation}\label{equation: Approx Opt No Wait Case One}
\optS\ge \timeReq^\opt_N+|\posReq^\opt_N-\posSch_N| \overset{(\ref{equation: Time First Request Last Schedule Opt})}{>} \timeSch_{N-1}+|\posReq^\opt_N-\posSch_N|.
\end{equation}
The inequality above gives us
\begin{align}
\smartstartS &\myoverset{(\ref{equation: Second Approx Upper Bound No Wait})}{<}{15}L(\timeSch_{N-1},\posSch_{N-1},\sigma_{S_{N-1}})+|\posReq^\opt_N-\posSch_N|+\optS\nonumber\\
&\myoverset{(\ref{equation: Approx Opt No Wait Case One})}{<}{15}L(\timeSch_{N-1},\posSch_{N-1},\sigma_{S_{N-1}})+2\optS-\timeSch_{N-1}\label{equation: Third Approx Upper Bound No Wait}
\end{align}
By definition of $\smartstart$, we have
\begin{equation}
\timeSch_{N-1}\ge\frac{1}{\Theta-1}L(\timeSch_{N-1},\posSch_{N-1},\sigma_{S_{N-1}}).\label{equation: Approx Time Second To Final Schedule}
\end{equation}
This leads to
\begin{align*}
\smartstartS &\myoverset{(\ref{equation: Third Approx Upper Bound No Wait})}{<}{25} L(\timeSch_{N-1},\posSch_{N-1},\sigma_{S_{N-1}})+2\optS-\timeSch_{N-1}\\
&\myoverset{(\ref{equation: Approx Time Second To Final Schedule})}{\le}{25} L(\timeSch_{N-1},\posSch_{N-1},\sigma_{S_{N-1}})+2\optS\\
&\qquad -\frac{1}{\Theta-1}L(\timeSch_{N-1},\posSch_{N-1},\sigma_{S_{N-1}})\\
&\myoverset{}{=}{25} \frac{\Theta-2}{\Theta-1}L(\timeSch_{N-1},\posSch_{N-1},\sigma_{S_{N-1}})+2\optS\\
&\myoverset{Lem \ref{lemma: Costs per Schedule}}{\le}{25} \frac{\Theta-2}{\Theta-1}\left(1+\frac{\Theta}{\Theta+2}\right)\optS+2\optS\\
&\myoverset{}{=}{25}\frac{4\Theta^2-8}{\Theta^2+\Theta-2}\optS.
\end{align*}
It can be shown that
\begin{equation*}
\frac{4\Theta^2-8}{\Theta^2+\Theta-2}<\Theta+1-\frac{\Theta-1}{3\Theta+3}
\end{equation*}
holds for all $\Theta>1$, which concludes this case.

\paragraph*{\textit{Case 2: }${\normalfont\opt}$\textit{ delivers }$\sigma^{\normalfont\textsc{Smart}}_{S_{N-1}}$\textit{ before collecting the request }$\sigma^{\normalfont\opt}_{S_N}$\\}

In this case we notice that we have
\begin{align*}
\smartstartS &\myoverset{(\ref{equation: First Approx Upper Bound No Wait})}{\le}{50} \timeSch_{N-1}+L(\timeSch_{N-1},\posSch_{N-1},\sigma_{S_{N-1}})+L(\timeReq_N^\opt,\posSch_N,\sigma_{S_N})\\
&\myoverset{(\ref{equation: Approx Time Second To Final Schedule})}{\le}{50}\Theta \timeSch_{N-1}+L(\timeReq_N^\opt,\posSch_N,\sigma_{S_N})\\
&\myoverset{(\ref{equation: Schedule Time})}{\le}{50}\Theta \timeSch_{N-1}+|\posSch_N-\posReq_N^\opt|+L(\timeReq_N^\opt,\posReq_N^\opt,\sigma_{S_N})\\
&\myoverset{(\ref{equation: First Approx Opt No Wait})}{\le}{50} \Theta \timeSch_{N-1}+|\posSch_N-\posReq_N^\opt|+\optS-\timeReq_N^\opt\\
&\myoverset{(\ref{equation: Time First Request Last Schedule Opt})}{<}{50} (\Theta-1) \timeReq_N^\opt+|\posSch_N-\posReq_N^\opt|+\optS\\
&\myoverset{$\timeReq_N^\opt\le\optS$}{\le}{50} \Theta\optS+|\posSch_N-\posReq_N^\opt|.
\end{align*}
This means the claim is shown if, we have
\begin{equation*}
|\posSch_N-\posReq_N^\opt|< \optS-\frac{\Theta-1}{3\Theta+3}\optS.
\end{equation*}
Therefore, we may assume in the following that
\begin{equation}
|\posSch_N-\posReq_N^\opt|\ge \optS-\frac{\Theta-1}{3\Theta+3}\optS.\label{equation: Distance Of Starting Points Assumption}
\end{equation}
Let $\optS=|x_-|+x_++y$ for some $y\ge 0$. By definition of $x_-$ and $x_+$ we have
\begin{equation}
|\posSch_N-\posReq_N^\opt|+y\le \optS.\label{equation: Distance Of Starting Points Approx One}
\end{equation}
Since by assumption $\opt$ delivers $\sigma^{\textsc{Smart}}_{S_{N-1}}$ to position~$\posSch_N$ before collecting $\sigma^{\opt}_{S_N}$ at position $\posReq^\opt_N$, we have
\begin{equation}
|\posSch_N-\posReq^\opt_N|+|\posSch_N|\le \optS ,\label{equation: Distance Of Starting Points Approx Two}
\end{equation}
and since $\sigma^{\textsc{Smart}}_{S_{N-1}}$ appears after time $\timeSch_{N-2}$, we also have
\begin{equation}
|\posSch_N-\posReq^\opt_N|+\timeSch_{N-2}< \optS. \label{equation: Distance Of Starting Points Approx Three}
\end{equation}
To sum it up, we may assume that
\begin{equation}
\max\{y,|\posSch_N|,\timeSch_{N-2}\}\overset{(\ref{equation: Distance Of Starting Points Assumption}),(\ref{equation: Distance Of Starting Points Approx One}),(\ref{equation: Distance Of Starting Points Approx Two}),(\ref{equation: Distance Of Starting Points Approx Three})}\le\frac{\Theta-1}{3\Theta+3}\optS\label{equation: Central Approx Upper Bound No Wait}
\end{equation}
holds. In the following, denote by $y^{S_{N-1}}_-$ the leftmost starting or ending point and by $y^{S_{N-1}}_+$ the rightmost starting or ending point of the requests in $\sigma_{S_{N-1}}$. We compute
\begin{align}
\smartstartS &\myoverset{(\ref{equation: Second Approx Upper Bound No Wait})}{<}{22} L(\timeSch_{N-1},\posSch_{N-1},\sigma_{S_{N-1}})+|\posSch_N-\posReq_N^\opt|+\optS\nonumber\\
&\myoverset{(\ref{equation: Distance Of Starting Points Approx Three})}{<}{22} L(\timeSch_{N-1},\posSch_{N-1},\sigma_{S_{N-1}})+2\optS-\timeSch_{N-2}\nonumber\\
&\myoverset{(\ref{equation: Schedule Time})}{\le}{22} |\posSch_{N-1}|+L(\timeSch_{N-1},0,\sigma_{S_{N-1}})+2\optS-\timeSch_{N-2}\nonumber\\
&\myoverset{Lem \ref{lemma: Lower Bound Starting Time}}{\le}{22} (\Theta-1) \timeSch_{N-2}+L(\timeSch_{N-1},0,\sigma_{S_{N-1}})+2\optS\nonumber\\
&\myoverset{Lem \ref{lemma: Approx Schedule From Zero}}{\le}{22} (\Theta-1) \timeSch_{N-2}+\max\{0,|y^{S_{N-1}}_-|\}+\max\{0,y^{S_{N-1}}_+\}\nonumber\\
&\qquad +y+2\optS.\label{equation: Fourth Approx Upper Bound No Wait}
\end{align}
Obviously, position $y^{S_{N-1}}_+$ is visited by $\smartstart$ in schedule $S_{N-1}$. Therefore, $y^{S_{N-1}}_+$ is smaller than or equal to the rightmost point that is visited by $\smartstart$ during schedule~$S_{N-1}$, which gives us
\begin{equation}
y^{S_{N-1}}_+\overset{\text{Lem }\ref{lemma: Rightmost Position}}{\le} |\posSch_{N-1}|+|\posSch_{N-1}-\posSch_N|+y-\max\{0,|y^{S_{N-1}}_-|\}.\label{equation: Upper Bound Rightmost Point}
\end{equation}
On the other hand, because of $|x_-|\le x_+$, we have $\optS\ge 2|x_-|+x_+$, which implies $y\ge |x_-|$. By definition of $x_-$ and $y^{S_{N-1}}_-$, we have $|x_-|\ge \max\{0,|y^{S_{N-1}}_-|\}$. This gives us~$y\ge\max\{0,|y^{S_{N-1}}_-|\}$ and
\begin{equation}
0\le |\posSch_{N-1}|+|\posSch_{N-1}-\posSch_N|+y-\max\{0,|y^{S_{N-1}}_-|\}.\label{equation: Approx Distance Starting Points Last Schedules}
\end{equation}
To sum it up, we have
\begin{equation}
\max\{0,y^{S_{N-1}}_+\}\overset{(\ref{equation: Upper Bound Rightmost Point}),(\ref{equation: Approx Distance Starting Points Last Schedules})}{\le} |\posSch_{N-1}|+|\posSch_{N-1}-\posSch_N|+y-\max\{0,|y^{S_{N-1}}_-|\}.\label{equation: Upper Bound Rightmost Point Or Zero}
\end{equation}
The inequality above gives us
\begin{align*}
\smartstartS &\myoverset{(\ref{equation: Fourth Approx Upper Bound No Wait})}{<}{24} (\Theta-1) \timeSch_{N-2}+\max\{0,|y^{S_{N-1}}_-|\}+\max\{0,y^{S_{N-1}}_+\}\nonumber\\
&\qquad +y+2\optS\\
&\myoverset{(\ref{equation: Upper Bound Rightmost Point Or Zero})}{\le}{24} (\Theta-1) \timeSch_{N-2}+|\posSch_{N-1}|+|\posSch_{N-1}-\posSch_N|+2y+2\optS\\
&\myoverset{}{\le}{24} (\Theta-1) \timeSch_{N-2}+|\posSch_{N-1}|+|\posSch_{N-1}|+|\posSch_N|+2y+2\optS\\
&\myoverset{Lem \ref{lemma: Lower Bound Starting Time}}{\le}{24} (\Theta-1) \timeSch_{N-2}+2\Theta \timeSch_{N-2}+|\posSch_N|+2y+2\optS\\
&\myoverset{(\ref{equation: Central Approx Upper Bound No Wait})}{\le}{24} (3\Theta+2)\frac{\Theta-1}{3\Theta+3}\optS+2\optS\\
&\myoverset{}{=}{24} \left(\Theta+1-\frac{\Theta-1}{3\Theta+3}\right)\optS.\qedhere
\end{align*}
\end{proof}

We combine the results of \Cref{proposition: Upper Bound Waiting} and \Cref{proposition: Upper Bound No Waiting} to obtain the main result of this section.

\begin{theorem}\label{theorem: General Upper Bound}
Let $\Theta^*$ be the only positive, real solution of $f_1(\Theta) = f_2(\Theta)$, i.e.,
\begin{equation*}
\Theta^*+1-\frac{\Theta^*-1}{3\Theta^*+3} = \frac{2\Theta^{*2}+2\Theta^*}{\Theta^{*2}+\Theta^*-2}.
\end{equation*}
Then, $\smartstart_{\Theta^*}$ is $\rho^*$-competitive with $\rho^* := f_1(\Theta^*) = f_2(\Theta^*) \approx 2.93768$.
\end{theorem}
\begin{proof}
For the case, where $\smartstart$ does wait before starting the final schedule, we have established the upper bound
\begin{equation*}
\frac{\smartstartS}{\optS}\le\frac{2\Theta^2+2\Theta}{\Theta^2+\Theta-2}=f_1(\Theta)
\end{equation*}
in \Cref{proposition: Upper Bound Waiting} and for the case, where $\smartstart$ starts the final schedule immediately after the second to final one, we have established the upper bound
\begin{equation*}
\frac{\smartstartS}{\optS}\le\Theta+1-\frac{\Theta-1}{3\Theta+3}=f_2(\Theta)
\end{equation*}
in \Cref{proposition: Upper Bound No Waiting}. Therefore the parameter for $\smartstart$ with the smallest upper bound~is
\begin{equation*}
\Theta^*=\argmin_{\Theta>1}\left\{\max\{f_1(\Theta),f_2(\Theta)\}\right\}.
\end{equation*}
 We note that $f_1$ is strictly decreasing for $\Theta>1$ and that $f_2$ is strictly increasing for $\Theta>1$. Therefore the minimum above lies at the intersection point of $f_1$ and $f_2$ that is larger than $1$, i.e., $\Theta^*$ is the only positive, real solution of
\begin{equation*}
\Theta+1-\frac{\Theta-1}{3\Theta+3}=\frac{2\Theta^2+2\Theta}{\Theta^2+\Theta-2}.
\end{equation*}
The resulting upper bound for the competitive ratio is
\begin{equation*}
\rho^*=f_1(\Theta^*)=f_2(\Theta^*)\approx 2.93768.\qedhere
\end{equation*}
\end{proof}

\section{Lower Bound for the Open Version}\label{section: Lower Bound for the Open Version}

In this section, we explicitly construct instances that demonstrate that the upper bounds given in the previous section are tight for certain ranges of $\Theta>1$, in particular for~$\Theta = \Theta^*$ (as in \Cref{theorem: General Upper Bound}). 
Further, we show that choices of $\Theta > 1$ different from $\Theta^*$ yield competitive ratios worse than~$\rho^*\approx 2.94$.
Together, this implies that $\rho^*$ is exactly the best possible competitive ratio for \smartstart.

All our lower bounds rely on the following lemma that gives a way to lure \smartstart away from the origin, with almost no time overhead.
More specifically, the lemma provides a way to make \smartstart move to any position~$p>0$ within time~$p+\mu$, where $\mu>0$ is arbitrarily small.

\begin{lemma}\label{lemma: Luring}
Let the capacity $c\in\N\cup\{\infty\}$ of the server be arbitrary but fixed, $p>0$ be any position on the real line and $\mu>0$ be any positive number. Furthermore, let $\delta>0$ be such that~$\frac{p}{\delta\Theta}=n\in\N$ and~$\delta<(\Theta-1)\mu$. Algorithm $\smartstart$ finishes serving the set of requests~$\sigma=\{\sigma_1,\dots,\sigma_{n+1}\}$ with
\begin{align*}
\sigma_1&=(\delta,\delta;0),\\
\sigma_i&=\left(i\delta,i\delta;\frac{1}{\Theta-1}\delta +(i-1)\delta\right)\text{ for }i\in\{2,\dots,n\}\\
\sigma_{n+1}&=\left(p,p;\mu +n\delta\right)
\end{align*}
and reaches the position $p$ at time $p+\mu$, provided that no additional requests appear until time~$\frac{p}{\Theta}+\mu$.
\end{lemma}

\begin{proof}
We show via induction that every request $\sigma_i$ with $i\in\{1,\dots,n\}$ is served in a separate schedule $S_i$ with starting position $\posSch_i=(i-1)\delta$ and starting time
\begin{equation*}
\timeSch_i=\frac{1}{\Theta-1}\delta +(i-1)\delta.
\end{equation*}
This is clear for $i=1$: By definition, $\smartstart$ starts from $\posSch_1=0$. The schedule~$S_1$ to serve $\sigma_1$ is started at time
\begin{equation*}
\timeSch_1=\min\left\{t\ge 0\;\middle|\; \frac{L(\timeSch,0,\{\sigma_1\})}{\Theta-1}\le t\right\}=\frac{1}{\Theta-1}\delta,
\end{equation*}
and reaches position $\delta$ at time $\frac{1}{\Theta-1}\delta+\delta=\frac{\Theta}{\Theta-1}\delta$. 
Note that the release time of every request~$\sigma_i$ is larger than $\timeSch_1$, ensuring that $S_1$ indeed only serves $\sigma_1$. 

We assume the claim is true for some $k\in\{1,\dots,n-1\}$. Consider $i=k+1$. By reduction, the server finishes schedule $S_k$ at position~$\posSch_{k+1}=k\delta$ at time~$\frac{1}{\Theta-1}\delta +k\delta$. Therefore, we have
\begin{equation*}
\timeSch_{k+1}\ge\frac{1}{\Theta-1}\delta+k\delta.
\end{equation*}
On the other hand, we have
\begin{equation*}
\frac{L\left(\frac{\delta}{\Theta-1} +k\delta,k\delta,\{\sigma_{k+1}\}\right)}{\Theta-1}=\frac{\delta}{\Theta-1}<\frac{1}{\Theta-1}\delta+k\delta.
\end{equation*}
Since there are no other unserved requests at time $\frac{1}{\Theta-1}\delta +k\delta$, the schedule $S_{k+1}$ is started at time~$\timeSch_{k+1}=\frac{1}{\Theta-1}\delta +k\delta$ and only serves~$\sigma_{k+1}$ as claimed.
It remains to examine the final request $\sigma_{n+1}$. The above shows that in the schedule $S_n$ is finished at time
\begin{equation*}
\timeSch_n+L(\timeSch_n,\posSch_n,\{\sigma_n\})=\frac{1}{\Theta-1}\delta +(n-1)\delta+\delta=\frac{1}{\Theta-1}\delta +n\delta<\mu+n\delta
\end{equation*}
at position $n\delta=\frac{p}{\Theta}$, i.e., before the request $\sigma_{n+1}$ is released at time $\mu+n\delta$. On the other hand, we have
\begin{equation*}
\frac{L\left(\mu +n\delta,\frac{p}{\Theta},\{\sigma_{n+1}\}\right)}{\Theta-1}=\frac{\frac{\Theta-1}{\Theta}p}{\Theta-1}=\frac{p}{\Theta}=n\delta<\mu+n\delta.
\end{equation*}
Therefore the final schedule $S_{n+1}$ is started at time~$\timeSch_{n+1}=\mu +n\delta=\mu +\frac{p}{\Theta}$, and we get
\begin{align*}
\smartstart((\sigma_i)_{i\in\{1,\dots,n+1\}})&=\timeSch_{n+1}+L(\timeSch_{n+1},\posSch_{n+1},\{\sigma_{n+1}\})\\
&=\mu +\frac{p}{\Theta}+\frac{\Theta-1}{\Theta}p\\
&=\mu +p.
\end{align*}
Note that for every request the starting point is identical to the ending point. 
Thus, our construction remains valid for every capacity $c\in\N\cup\{\infty\}$. 
Furthermore, there is no interference with requests that are released after time $\timeSch_{n+1}=\mu +\frac{p}{\Theta}$.
\end{proof}

Equipped with this strategy to lure $\smartstart$ away from the origin, we now move on to establish lower bounds matching Propositions~\ref{proposition: Upper Bound Waiting} and~\ref{proposition: Upper Bound No Waiting}.

\begin{proposition}\label{proposition: Lower Bound Waiting}
Let the capacity $c\in\N\cup\{\infty\}$ of the server be arbitrary but fixed and let $2<\Theta<3$.
For every sufficiently small $\varepsilon>0$, there is a set of requests $\sigma$ such that $\smartstart$ waits before starting the final schedule and such that the inequality
\begin{equation*}
\frac{\smartstartS}{\optS}\ge\frac{2\Theta^2+2\Theta}{\Theta^2+\Theta-2}-\varepsilon
\end{equation*}
holds, i.e., the upper bound established in \Cref{proposition: Upper Bound Waiting} is tight for $\Theta\in(2,3)$.
\end{proposition}

\begin{proof}
Let $\varepsilon>0$ with $\smash{\varepsilon<\frac{2}{9}(\frac{2\Theta^2}{\Theta^2+\Theta-2})}$ and $\smash{\varepsilon'=\frac{\Theta^2+\Theta-2}{2\Theta^2}\varepsilon}$. We apply \Cref{lemma: Luring} with $p=1$ and~$\mu=\frac{\varepsilon'}{2}$. For convenience, we start the enumeration of schedules with the first schedule after the application of \Cref{lemma: Luring}. $\smartstart$ reaches position $\posSch_1=1$ at time $1+\frac{\varepsilon'}{2}$. Now let the requests 
\begin{align*}
\sigma_1^{(1)}&=\left(-\frac{1}{\Theta}+\varepsilon',0;\frac{1}{\Theta}+\varepsilon'\right),\\
\sigma_1^{(2)}&=\left(\frac{1}{\Theta},1;\frac{1}{\Theta}+\varepsilon'\right)
\end{align*}
appear. Note that both requests appear after time $\frac{1}{\Theta}+\frac{\varepsilon'}{2}$ and therefore do not interfere with the application of \Cref{lemma: Luring}. If $\smartstart$ delivers $\smash{\sigma_1^{(2)}}$ before collecting~$\smash{\sigma_1^{(1)}}$ the time it needs is at least
\begin{align*}
&\phantom{=}\;\left|1-\frac{1}{\Theta}\right|+\left|\frac{1}{\Theta}-1\right|+\left|1-\left(-\frac{1}{\Theta}+\varepsilon'\right)\right|+\left|\left(-\frac{1}{\Theta}+\varepsilon'\right)-0\right|\\
&=\frac{2\Theta-2}{\Theta}+\frac{\Theta+2}{\Theta}-2\varepsilon'\\
&=3-2\varepsilon'.
\end{align*}
The best schedule that delivers $\smash{\sigma_1^{(2)}}$ after collecting $\smash{\sigma_1^{(1)}}$ delivers $\smash{\sigma_1^{(1)}}$ before visiting the starting point $-\frac{1}{\Theta}+\varepsilon'$ of $\smash{\sigma_1^{(2)}}$ and needs time
\begin{align*}
&\phantom{=}\;\left|1-\left(-\frac{1}{\Theta}+\varepsilon'\right)\right|+\left|\left(-\frac{1}{\Theta}+\varepsilon'\right)-0\right|+\left|0-\frac{1}{\Theta}\right|+\left|\frac{1}{\Theta}-1\right|\\
&=1+\frac{1}{\Theta}-\varepsilon'+\frac{1}{\Theta}-\varepsilon'+1\\
&=2+\frac{2}{\Theta}-2\varepsilon'.
\end{align*}
By assumption, we have $\Theta>2$, which implies $2+\frac{2}{\Theta}-2\varepsilon'<3-2\varepsilon'$. Therefore, $\smartstart$ delivers $\sigma_1^{(2)}$ after collecting $\sigma_1^{(1)}$ and, for all $t\ge 1+\frac{\varepsilon'}{2}$, we have
\begin{equation*}
L(t,\posSch_1,\{\sigma_1^{(1)},\sigma_1^{(2)}\})=L(t,1,\{\sigma_1^{(1)},\sigma_1^{(2)}\})=2+\frac{2}{\Theta}-2\varepsilon'.
\end{equation*}
Again, by assumption, we have $\Theta<3$ and $\smash{\varepsilon<\frac{2}{9}(\frac{2\Theta^2}{\Theta^2+\Theta-2})}$, i.e., $\smash{\varepsilon'<\frac{2}{9}}$, which implies that for the time~$1+\frac{\varepsilon'}{2}$, when $\smartstart$ reaches position $\posSch_1=1$, the inequality
\begin{equation}
\frac{L(1+\frac{\varepsilon'}{2},\posSch_1,\{\sigma_1^{(1)},\sigma_1^{(2)}\})}{\Theta-1}=\frac{2+\frac{2}{\Theta}-2\varepsilon'}{\Theta-1}\overset{\Theta<3}{>}1+\frac{1}{3}-\varepsilon'\overset{\varepsilon'<\frac{2}{9}}{>}1+\frac{\varepsilon'}{2}\label{equation: Lower Bound Approx Starting Time One}
\end{equation}
holds. (Note that inequality~(\ref{equation: Lower Bound Approx Starting Time One}) also holds for slightly larger $\Theta$ if we let $\varepsilon\rightarrow 0$.) Because of inequality~(\ref{equation: Lower Bound Approx Starting Time One}), $\smartstart$ has a waiting period and starts the schedule~$S_1$ at time
\begin{align*}
\timeSch_1&=\min\left\{t\ge 1+\frac{\varepsilon'}{2}\;\middle|\; \frac{L(t,\posSch_1,\{\sigma_1^{(1)},\sigma_1^{(2)}\})}{\Theta-1}\le t\right\}\\
&=\min\left\{t\ge 1+\frac{\varepsilon'}{2}\;\middle|\; \frac{2+\frac{2}{\Theta}-2\varepsilon'}{\Theta-1}\le t\right\}\\
&=\frac{2+\frac{2}{\Theta}-2\varepsilon'}{\Theta-1}\\
&=\frac{2\Theta+2-2\varepsilon'\Theta}{\Theta(\Theta-1)}.
\end{align*}
To sum it up, we have
\begin{align*}
\smartstartS &=\timeSch_1+L(\timeSch_1,\posSch_1,\{\sigma_1^{(1)},\sigma_1^{(2)}\})\\
&=\frac{2\Theta+2-2\varepsilon'\Theta}{\Theta(\Theta-1)}+2+\frac{2}{\Theta}-2\varepsilon'\\
&=\frac{2\Theta+2-2\varepsilon'\Theta}{\Theta-1}.
\end{align*}
On the other hand, $\opt$ goes from the origin to $-\frac{1}{\Theta}+\varepsilon'$ to collect $\sigma_1^{(1)}$ at time $\frac{1}{\Theta}+\varepsilon'$ (i.e., it has to wait for $2\varepsilon'$ units of time after it reaches position $-\frac{1}{\Theta}+\varepsilon'$) and delivers $\sigma_1^{(1)}$ to the origin at time~$\frac{2}{\Theta}$. Let $q>0$ be the position of a request arising from the application of \Cref{lemma: Luring} at the beginning of this proof. Then this requests is released earlier than time $q+\frac{\varepsilon'}{2}$. On the other hand, $\opt$ reaches position~$q$ not earlier than time~$\frac{2}{\Theta}+q$. Since we have $\Theta<3$ and $\smash{\varepsilon<\frac{2}{9}(\frac{2\Theta^2}{\Theta^2+\Theta-2})}$, i.e., $\varepsilon'<\frac{2}{9}$, we have~$\smash{\frac{2}{\Theta}+q>q+\frac{\varepsilon'}{2}}$ and $\opt$ can go straight from the origin to position~$1$, collecting an delivering all requests that occur by the application of \Cref{lemma: Luring} as well as $\smash{\sigma_1^{(2)}}$ on the way. Therefore, we have
\begin{equation*}
\optS=\left|0-\left(-\frac{1}{\Theta}+\varepsilon'\right)\right|+2\varepsilon'+\left|\left(-\frac{1}{\Theta}+\varepsilon'\right)-1\right|=\frac{1}{\Theta}+\frac{\Theta+1}{\Theta}=\frac{\Theta+2}{\Theta}.
\end{equation*}
Note, that $\opt$ can do this even if the capacity is $c=1$, since no additional requests need to be carried over $[0,\frac{1}{\Theta}]\cup\{1\}$, where the requests of the application of \Cref{lemma: Luring} appear, and because the carrying paths of $\smash{\sigma_1^{(1)}}$ and $\smash{\sigma_1^{(2)}}$ are disjoint. Since we have~$\smash{\varepsilon'=\frac{\Theta^2+\Theta-2}{2\Theta^2}\varepsilon}$, we obtain
\begin{equation*}
\frac{\smartstartS}{\optS}=\frac{2\Theta^2+2\Theta-2\varepsilon'\Theta^2}{\Theta^2+\Theta-2}=\frac{2\Theta^2+2\Theta}{\Theta^2+\Theta-2}-\varepsilon,
\end{equation*}
as claimed.
\end{proof}

\begin{proposition}\label{proposition: Lower Bound No Waiting}
Let the capacity $c\in\N\cup\{\infty\}$ of the server be arbitrary but fixed and let $2\le\Theta\le\frac{1}{2}(1+\sqrt{13})$. For every sufficiently small $\varepsilon>0$ there is a set of requests $\sigma$ such that $\smartstart$ immediately starts $S_N$ after $S_{N-1}$ and such that
\begin{equation*}
\frac{\smartstartS}{\optS}\ge\Theta+1-\frac{\Theta-1}{3\Theta+3}-\varepsilon,
\end{equation*}
i.e., the upper bound established in \Cref{proposition: Upper Bound No Waiting} is tight for $\Theta\in[2,\frac{1}{2}(1+\sqrt{13})]\approx[2,2.303]$.
\end{proposition}

\begin{proof}
Let $\varepsilon>0$ with $\smash{\varepsilon<\frac{1}{5\Theta}\frac{3\Theta^2-\Theta}{3\Theta+3}}$ and $\smash{\varepsilon'=\frac{3\Theta+3}{3\Theta^2-\Theta}\varepsilon}$. We apply \Cref{lemma: Luring} with~$p=1$ and~$\mu=\frac{\varepsilon'}{2}$. For convenience, we start the enumeration of the schedules with the first schedule after the application of \Cref{lemma: Luring}. Algorithm $\smartstart$ reaches position $\posSch_1=1$ at time~$1+\frac{\varepsilon'}{2}$. Now let the requests 
\begin{align*}
\sigma_1^{(1)}&=\left(2+\frac{1}{\Theta}-\varepsilon',2+\frac{1}{\Theta}-\varepsilon';\frac{1}{\Theta}+\varepsilon'\right),\\
\sigma_1^{(2)}&=\left(-\frac{1}{\Theta},-\frac{1}{\Theta};\frac{1}{\Theta}+\varepsilon'\right)
\end{align*}
appear. 
Note that both requests are released after time $\frac{1}{\Theta}+\frac{\varepsilon'}{2}$ and, therefore, do not interfere with the application of \Cref{lemma: Luring}. 
If $\smartstart$ serves $\sigma_1^{(2)}$ before serving $\sigma_1^{(1)}$ the time it needs is at least
\begin{equation*}
\left|1-\left(-\frac{1}{\Theta}\right)\right|+\left|\left(-\frac{1}{\Theta}\right)-\left(2+\frac{1}{\Theta}-\varepsilon'\right)\right|=1+\frac{1}{\Theta}+2+\frac{2}{\Theta}-\varepsilon'=3+\frac{3}{\Theta}-\varepsilon'.
\end{equation*}
The best schedule that serves $\sigma_1^{(2)}$ after serving $\sigma_1^{(1)}$ needs time
\begin{align*}
\left|1-\left(2+\frac{1}{\Theta}-\varepsilon'\right)\right|+\left|\left(2+\frac{1}{\Theta}-\varepsilon'\right)-\left(-\frac{1}{\Theta}\right)\right|&=1+\frac{1}{\Theta}-\varepsilon'+2+\frac{2}{\Theta}-\varepsilon'\\
&=3+\frac{3}{\Theta}-2\varepsilon'.
\end{align*}
Thus, $\smartstart$ serves $\sigma_1^{(2)}$ after serving $\sigma_1^{(1)}$, and, for all $t\ge 1+\frac{\varepsilon'}{2}$, we obtain
\begin{equation*}
L\left(\timeSch,\posSch_1,\{\sigma_1^{(1)},\sigma_1^{(2)}\}\right)=L\left(\timeSch,1,\{\sigma_1^{(1)},\sigma_1^{(2)}\}\right)=3+\frac{3}{\Theta}-2\varepsilon'.
\end{equation*}
By assumption, we have $\Theta\le\frac{1}{2}(1+\sqrt{13})$ and $\varepsilon<\frac{1}{5\Theta}\frac{3\Theta^2-\Theta}{3\Theta+3}$, i.e., $\varepsilon'<\frac{1}{5\Theta}<1$, which implies that for the time $1+\frac{\varepsilon'}{2}$, when $\smartstart$ reaches position $\posSch_1=1$, the inequality
\begin{align*}
\frac{L\left(1+\frac{\varepsilon'}{2},\posSch_1,\{\sigma_1^{(1)},\sigma_1^{(2)}\}\right)}{\Theta-1}
&\myoverset{}{=}{64}\frac{3+\frac{3}{\Theta}-2\varepsilon'}{\Theta-1}\\
&\myoverset{}{=}{64}\frac{3-2\varepsilon'}{\Theta-1}+\frac{3}{\Theta(\Theta-1)}\\
&\myoverset{$1<\Theta\le\frac{1}{2}(1+\sqrt{13})$}{\ge}{64}\frac{3-2\varepsilon'}{\frac{1}{2}(\sqrt{13}-1)}+\frac{3}{\frac{1}{4}(\sqrt{13}-1)(1+\sqrt{13})}\\
&\myoverset{}{=}{64}\frac{3-2\varepsilon'}{\frac{1}{2}(\sqrt{13}-1)}+1\\
&\myoverset{$\frac{1}{2}(\sqrt{13}-1)<2$}{>}{64} 1 + \frac{\varepsilon'}{2}
\end{align*}
holds. 
Thus, $\smartstart$ has a waiting period and starts schedule $S_1$ at time
\begin{align*}
\timeSch_1&=\min\left\{t\ge 1+\frac{\varepsilon'}{2}\;\middle|\; \frac{L(\timeSch,\posSch_1,\{\sigma_1^{(1)},\sigma_1^{(2)}\})}{\Theta-1}\le t\right\}\\
&=\min\left\{t\ge 1+\frac{\varepsilon'}{2}\;\middle|\; \frac{3+\frac{3}{\Theta}-2\varepsilon'}{\Theta-1}\le t\right\}\\
&=\frac{3+\frac{3}{\Theta}-2\varepsilon'}{\Theta-1}\\
&=\frac{3\Theta+3}{\Theta(\Theta-1)}-\frac{2\varepsilon'}{\Theta-1}.
\end{align*}
Next, we let the final request
\begin{equation*}
\sigma_2=\left(\frac{3\Theta+3}{\Theta(\Theta-1)}-\frac{2}{\Theta}-\varepsilon',\frac{3\Theta+3}{\Theta(\Theta-1)}-\frac{2}{\Theta}-\varepsilon';\frac{3\Theta+3}{\Theta(\Theta-1)}\right)
\end{equation*}
appear. $\smartstart$ finishes schedule $S_1$ at time
\begin{equation*}
\timeSch_1+L(\timeSch_1,\posSch_1,\{\sigma_1^{(1)},\sigma_1^{(2)}\})=\frac{3\Theta+3}{\Theta(\Theta-1)}-\frac{2\varepsilon'}{\Theta-1}+3+\frac{3}{\Theta}-2\varepsilon'=\frac{3\Theta+3}{\Theta-1}-\frac{2\Theta\varepsilon'}{\Theta-1}
\end{equation*}
at position $\posSch_2=-\frac{1}{\Theta}$. For all $t\ge \frac{3\Theta+3}{\Theta-1}-\frac{2\Theta}{\Theta-1}\varepsilon'$, we obtain
\begin{equation*}
L\left(\timeSch,-\frac{1}{\Theta},\{\sigma_2\}\right)=\frac{3\Theta+3}{\Theta(\Theta-1)}-\frac{1}{\Theta}-\varepsilon'.
\end{equation*}
By assumption, we have~$2\le\Theta\le\frac{1}{2}(1+\sqrt{13}) < 3$ and~$\varepsilon<\frac{1}{5\Theta}\frac{3\Theta^2-\Theta}{3\Theta+3}$, i.e.,~$\varepsilon'<\frac{1}{5\Theta}$, which implies that, for the finishing time $\frac{3\Theta+3}{\Theta-1}-\frac{2\Theta\varepsilon'}{\Theta-1}$ of schedule~$S_1$, the inequality
\begin{align}
\frac{L\left(\frac{3\Theta+3}{\Theta-1}-\frac{2\Theta\varepsilon'}{\Theta-1},-\frac{1}{\Theta},\{\sigma_2\}\right)}{\Theta-1}
&\myoverset{}{=}{40}\frac{3\Theta+3}{\Theta(\Theta-1)^2}-\frac{1+\Theta\varepsilon'}{\Theta(\Theta-1)}\nonumber\\
&\myoverset{$\Theta\ge2$}{<}{40} \frac{3\Theta+3}{\Theta-1}-\frac{1+\Theta\varepsilon'}{\Theta(\Theta-1)}\nonumber\\
&\myoverset{$1>5\Theta\varepsilon'$}{<}{40}\frac{3\Theta+3}{\Theta-1}-\frac{6\varepsilon'}{\Theta-1}\nonumber\\
&\myoverset{$\Theta<3$}{<}{40}\frac{3\Theta+3}{\Theta-1}-\frac{2\Theta\varepsilon'}{\Theta-1}\label{equation: Starting Time Approximation Lower Bound}
\end{align}
holds. 
(Note that inequality (\ref{equation: Starting Time Approximation Lower Bound}) still holds for slightly smaller $\Theta$ if we let $\varepsilon\rightarrow 0$.) 
Because of inequality~(\ref{equation: Starting Time Approximation Lower Bound}), the final schedule $S_2$ is started at time 
\begin{equation*}
\timeSch_2=\frac{3\Theta+3}{\Theta-1}-\frac{2\Theta\varepsilon'}{\Theta-1}
\end{equation*}
without waiting. To sum it up, we have
\begin{align*}
\smartstartS&=\timeSch_2+L(\timeSch_2,\posSch_2,\{\sigma_2\})\\
&=\frac{3\Theta+3}{\Theta-1}-\frac{2\Theta\varepsilon'}{\Theta-1}+\frac{3\Theta+3}{\Theta(\Theta-1)}-\frac{1}{\Theta}-\varepsilon'\\
&=\frac{3\Theta+3}{\Theta-1}+\frac{3\Theta+3}{\Theta(\Theta-1)}-\frac{1}{\Theta}-\frac{3\Theta-1}{\Theta-1}\varepsilon'.
\end{align*}
On the other hand, $\opt$ goes from the origin straight to position $-\frac{1}{\Theta}$ serving request $\sigma_1^{(2)}$ at time $\frac{1}{\Theta}+\varepsilon'$ (i.e., it has to wait for $\varepsilon'$ units of time after it reaches position $-\frac{1}{\Theta}$) and returns to the origin at time~$\frac{2}{\Theta}+\varepsilon'$. 
Let $q>0$ be the position of a request that has occurred by the application of \Cref{lemma: Luring} at the beginning of this proof. 
Then this request is released earlier than time~$q+\frac{\varepsilon'}{2}$. Since $\opt$ reaches position $q$ not earlier than time~$\frac{2}{\Theta}+\varepsilon'+q>q+\frac{\varepsilon'}{2}$, $\opt$ can go straight from the origin to the right and can serve all remaining requests without waiting. 
Note that the position $\smash{\frac{3\Theta+3}{\Theta(\Theta-1)}-\frac{2}{\Theta}-\varepsilon'}$ of $\sigma_2$ is equal to or to right of the position $\smash{2+\frac{1}{\Theta}-\varepsilon'}$ of $\smash{\sigma_1^{(2)}}$ because of~$\Theta\le\frac{1}{2}(1+\sqrt{13})$. 
Thus, $\opt$ finishes at position~$\frac{3\Theta+3}{\Theta(\Theta-1)}-\frac{2}{\Theta}-\varepsilon'$ and we have
\begin{align*}
\optS&=\left|0-\left(-\frac{1}{\Theta}\right)\right|+\varepsilon'+\left|-\frac{1}{\Theta}-\left(\frac{3\Theta+3}{\Theta(\Theta-1)}-\frac{2}{\Theta}-\varepsilon'\right)\right|\\
&=\frac{1}{\Theta}+\varepsilon'+\frac{1}{\Theta}+\frac{3\Theta+3}{\Theta(\Theta-1)}-\frac{2}{\Theta}-\varepsilon'\\
&=\frac{3\Theta+3}{\Theta(\Theta-1)}.
\end{align*}
Note that $\opt$ can do this even if $c=1$ since for all requests the starting point is equal to the destination. 
Since we have $\varepsilon'=\frac{3\Theta+3}{3\Theta^2-\Theta}\varepsilon$, we finally obtain
\begin{align*}
\frac{\smartstartS}{\optS}&=\frac{\frac{3\Theta+3}{\Theta-1}+\frac{3\Theta+3}{\Theta(\Theta-1)}-\frac{1}{\Theta}-\frac{3\Theta-1}{\Theta-1}\varepsilon'}{\frac{3\Theta+3}{\Theta(\Theta-1)}}\\
&=\Theta+1-\frac{\Theta-1}{3\Theta+3}-\frac{3\Theta^2-\Theta}{3\Theta+3}\varepsilon'\\
&=\Theta+1-\frac{\Theta-1}{3\Theta+3}-\varepsilon,
\end{align*}
as claimed.
\end{proof}

Recall that the optimal parameter $\Theta^*$ established in \Cref{theorem: General Upper Bound} is the only positive, real solution of the equation
\begin{equation*}
\Theta+1-\frac{\Theta-1}{3\Theta+3}=\frac{2\Theta^2+2\Theta}{\Theta^2+\Theta-2},
\end{equation*}
which is $\Theta^*\approx 2.0526$. Therefore, according to \Cref{proposition: Lower Bound Waiting} and \Cref{proposition: Lower Bound No Waiting} the parameter $\Theta^*$ lies in the ranges where the upper bounds of Propositions \ref{proposition: Upper Bound Waiting} and \ref{proposition: Upper Bound No Waiting} are both tight. It remains to make sure that for all $\Theta$ that lie outside of this range the competitive ratio of $\smartstart_\Theta$ is larger than $\rho^*\approx 2.93768$.

\begin{lemma}\label{lemma: Lower Bound One}
Let the capacity $c\in\N\cup\{\infty\}$ of the server be arbitrary but fixed and let $1<\Theta\le 2$. There is a set of requests $\sigma$ such that
\begin{equation*}
\frac{\smartstartS}{\optS}>\rho^*\approx 2.93768.
\end{equation*}
\end{lemma}
\begin{proof}
Let~$\varepsilon>0$ with~$\varepsilon<\frac{1}{100}$. We apply \Cref{lemma: Luring} with~$p=1$ and $\mu=\varepsilon$. For convenience, we start the enumeration of the schedules with the first schedule after the application of \Cref{lemma: Luring}. $\smartstart$ reaches position $\posSch_1=1$ at time~$1+\varepsilon$. Now let the requests 
\begin{align*}
\sigma_1^{(1)}&=\left(\frac{1}{\Theta},1+\frac{\varepsilon}{2};\frac{1}{\Theta}+2\varepsilon\right),\\
\sigma_1^{(2)}&=\left(-\frac{1}{\Theta},-\varepsilon;\frac{1}{\Theta}+2\varepsilon\right)
\end{align*}
appear. Note that both requests appear after time $\frac{1}{\Theta}+\varepsilon$ and therefore do not interfere with the application of \Cref{lemma: Luring}. If $\smartstart$ collects $\sigma_1^{(2)}$ before delivering $\sigma_1^{(1)}$ the time it needs is at least
\begin{equation*}
\left|1-\left(-\frac{1}{\Theta}\right)\right|+\left|\left(-\frac{1}{\Theta}\right)-\left(1+\frac{\varepsilon}{2}\right)\right|=2+\frac{2}{\Theta}+\frac{\varepsilon}{2}.
\end{equation*}
The best schedule that collects $\sigma_1^{(2)}$ after delivering $\sigma_1^{(1)}$ needs time
\begin{align*}
&\hphantom{=}\,\left|1-\frac{1}{\Theta}\right|+\left|\frac{1}{\Theta}-\left(1+\frac{\varepsilon}{2}\right)\right|+\left|\left(1+\frac{\varepsilon}{2}\right)-\left(-\frac{1}{\Theta}\right)\right|+\left|\left(-\frac{1}{\Theta}\right)-(-\varepsilon)\right|\\
&=\frac{\Theta-1}{\Theta}+\frac{\Theta-1}{\Theta}+\frac{\varepsilon}{2}+\frac{\Theta+1}{\Theta}+\frac{\varepsilon}{2}+\frac{1}{\Theta}-\varepsilon\\
&=3.
\end{align*}
By assumption, we have $\Theta\le 2$, which implies $3<2+\frac{2}{\Theta}+\frac{\varepsilon}{2}$. Therefore, $\smartstart$ delivers $\sigma_1^{(2)}$ after collecting $\sigma_1^{(1)}$ and for all $t\ge 1+\varepsilon$ we have
\begin{equation*}
L(t,\posSch_1,\{\sigma_1^{(1)},\sigma_1^{(2)}\})=L(t,1,\{\sigma_1^{(1)},\sigma_1^{(2)}\})=3.
\end{equation*}
Again, by assumption, we have $\Theta\le 2$ and $\varepsilon<\frac{1}{100}$, which implies that for the time~$1+\varepsilon$, when $\smartstart$ reaches position $\posSch_1=1$, the inequality
\begin{equation*}
\frac{L\left(1+\varepsilon,\posSch_1,\{\sigma_1^{(1)},\sigma_1^{(2)}\}\right)}{\Theta-1}=\frac{3}{\Theta-1}>1+\varepsilon
\end{equation*}
holds. Thus, $\smartstart$ has a waiting period and starts schedule $S_1$ at time
\begin{align*}
\timeSch_1&=\min\left\{t\ge 1+\varepsilon\;\middle|\; \frac{L(t,\posSch_1,\{\sigma_1^{(1)},\sigma_1^{(2)}\})}{\Theta-1}\le t\right\}\\
&=\min\left\{t\ge 1+\varepsilon\;\middle|\; \frac{3}{\Theta-1}\le t\right\}\\
&=\frac{3}{\Theta-1}.
\end{align*}
To sum it up, we have
\begin{equation*}
\smartstartS=\timeSch_1+L(\timeSch_1,\posSch_1,\{\sigma_1^{(1)},\sigma_1^{(2)}\})=\frac{3}{\Theta-1}+3=\frac{3\Theta}{\Theta-1}.
\end{equation*}
On the other hand, $\opt$ goes from the origin to $-\frac{1}{\Theta}$ to collect $\sigma_1^{(1)}$ at time $\frac{1}{\Theta}+2\varepsilon$ (i.e., it has to wait for $2\varepsilon$ units of time after it reaches position $-\frac{1}{\Theta}$) and returns to the origin at time $\frac{2}{\Theta}+2\varepsilon$. Let $q$ be the position of a request that has occurred by the application of \Cref{lemma: Luring} at the beginning of this proof. Then this requests is released earlier than time $q+\varepsilon$. Since $\opt$ reaches position $q$ not earlier than time~$\frac{2}{\Theta}+2\varepsilon+q>q+\varepsilon$, $\opt$ can go straight from position $-\frac{1}{\Theta}$ to position $1+\frac{\varepsilon}{2}$ collecting and delivering all requests that occur by the application of \Cref{lemma: Luring} as well as $\smash{\sigma_1^{(2)}}$ on the way. Therefore, we have
\begin{equation*}
\optS=\left|0-\left(-\frac{1}{\Theta}\right)\right|+2\varepsilon+\left|-\frac{1}{\Theta}-\left(1+\frac{\varepsilon}{2}\right)\right|=\frac{2}{\Theta}+1+\frac{5\varepsilon}{2}=\frac{\Theta+2+\frac{5}{2}\varepsilon\Theta}{\Theta}.
\end{equation*}
Note, that $\opt$ can do this even if the capacity is $c=1$ since no additional requests need to be carried over $[0,\frac{1}{\Theta}]\cup\{1\}$, where the requests of the application of \Cref{lemma: Luring} appear, and since the carrying paths of $\smash{\sigma_1^{(1)}}$ and $\smash{\sigma_1^{(2)}}$ are disjoint. Finally, we have
\begin{equation}\label{equation: Fraction Smartstart Opt One}
\frac{\smartstartS}{\optS}=\frac{3\Theta^2}{\Theta^2+\Theta-2+\frac{5}{2}\varepsilon(\Theta^2-\Theta)}.
\end{equation}
Note that the fraction in equality (\ref{equation: Fraction Smartstart Opt One}) becomes larger with decreasing $\varepsilon$. By assumption, we have $\varepsilon<\frac{1}{100}$, which implies
\begin{equation*}
\frac{\smartstartS}{\optS}>\frac{3\Theta^2}{\frac{41}{40}\Theta^2+\frac{39}{40}\Theta-2}=:g_1(\Theta).
\end{equation*}
The function $g_1$ is monotonically decreasing on $(1,2]$. Therefore, we have
\begin{equation*}
\frac{\smartstartS}{\optS}>g_1(2)=\frac{80}{27}>2.95>\rho^*
\end{equation*}
for all $\Theta\in (1,2]$.
\end{proof}

\begin{lemma}\label{lemma: Lower Bound Two}
Let the capacity $c\in\N\cup\{\infty\}$ of the server be arbitrary but fixed and let $\frac{1}{2}(1+\sqrt{13})<\Theta\le 1+\sqrt{2}$. There is a set of requests $\sigma$ such that
\begin{equation*}
\frac{\smartstartS}{\optS}>\rho^*\approx 2.93768.
\end{equation*}
\end{lemma}
\begin{proof}
Let $\varepsilon>0$ with $\varepsilon<\frac{1}{25}$. We apply \Cref{lemma: Luring} with $p=1$ and $\mu=\varepsilon$. For convenience, we start the enumeration of the schedules with the first schedule after the application of \Cref{lemma: Luring}. $\smartstart$ reaches position $\posSch_1=1$ at time $1+\varepsilon$. Now let the requests 
\begin{align*}
\sigma_1^{(1)}&=\left(2+\frac{1}{\Theta}-\frac{\varepsilon}{2},2+\frac{1}{\Theta}-\frac{\varepsilon}{2};\frac{1}{\Theta}+2\varepsilon\right),\\
\sigma_1^{(2)}&=\left(-\frac{1}{\Theta}-\varepsilon,-\frac{1}{\Theta}-\varepsilon;\frac{1}{\Theta}+2\varepsilon\right)
\end{align*}
appear. Note that both requests appear after time $\frac{1}{\Theta}+\varepsilon$ and therefore do not interfere with the application of \Cref{lemma: Luring}. If $\smartstart$ serves $\sigma_1^{(2)}$ before serving $\sigma_1^{(1)}$ the time it needs is at least
\begin{align*}
\left|1-\left(-\frac{1}{\Theta}-\varepsilon\right)\right|+\left|\left(-\frac{1}{\Theta}-\varepsilon\right)-\left(2+\frac{1}{\Theta}-\frac{\varepsilon}{2}\right)\right|&=\frac{\Theta+1}{\Theta}+\varepsilon+\frac{2\Theta+2}{\Theta}+\frac{\varepsilon}{2}\\
&=3+\frac{3}{\Theta}+\frac{3}{2}\varepsilon.
\end{align*}
The best schedule that serves $\sigma_1^{(2)}$ after serving $\sigma_1^{(1)}$ needs time
\begin{align*}
&\phantom{=}\;\left|1-\left(2+\frac{1}{\Theta}-\frac{\varepsilon}{2}\right)\right|+\left|\left(2+\frac{1}{\Theta}-\frac{\varepsilon}{2}\right)-\left(-\frac{1}{\Theta}-\varepsilon\right)\right|\\
&=\frac{\Theta+1}{\Theta}-\frac{\varepsilon}{2}+\frac{2\Theta+2}{\Theta}+\frac{\varepsilon}{2}\\
&=3+\frac{3}{\Theta}.
\end{align*}
Therefore $\smartstart$ serves $\sigma_1^{(2)}$ after serving $\sigma_1^{(1)}$ and, for all $t\ge 1+\varepsilon$, we have
\begin{equation*}
L(t,\posSch_1,\{\sigma_1^{(1)},\sigma_1^{(2)}\})=L(t,1,\{\sigma_1^{(1)},\sigma_1^{(2)}\})=3+\frac{3}{\Theta}.
\end{equation*}
By assumption, we have $\Theta\le 1+\sqrt{2}$ and $\varepsilon<\frac{1}{25}$, which implies that for the time $1+\varepsilon$, when $\smartstart$ reaches position $\posSch_1=1$, the inequality
\begin{align*}
\frac{L\left(1+\varepsilon,\posSch_1,\{\sigma_1^{(1)},\sigma_1^{(2)}\}\right)}{\Theta-1}&\myoverset{}{=}{55}\frac{3+\frac{3}{\Theta}}{\Theta-1}\\
&\myoverset{$\Theta\le 1+\sqrt{2}$}{\ge}{55}\frac{3}{\sqrt{2}}+\frac{3}{2+\sqrt{2}}\\
&\myoverset{$\varepsilon<\frac{1}{10}$}{>}{55}1+\frac{\varepsilon}{2}
\end{align*}
holds. Thus, $\smartstart$ has a waiting period and starts schedule $S_1$ at time
\begin{align*}
\timeSch_1&=\min\left\{t\ge 1+\varepsilon\;\middle|\; \frac{L(t,\posSch_1,\{\sigma_1^{(1)},\sigma_1^{(2)}\})}{\Theta-1}\le t\right\}\\
&=\min\left\{t\ge 1+\varepsilon\;\middle|\; \frac{3+\frac{3}{\Theta}}{\Theta-1}\le t\right\}\\
&=\frac{3+\frac{3}{\Theta}}{\Theta-1}\\
&=\frac{3\Theta+3}{\Theta(\Theta-1)}.
\end{align*}
Next, we let the final request
\begin{equation*}
\sigma_2=\left(2+\frac{1}{\Theta}-\varepsilon,2+\frac{1}{\Theta}-\varepsilon;\frac{3\Theta+3}{\Theta(\Theta-1)}+\varepsilon\right)
\end{equation*}
appear. $\smartstart$ finishes schedule $S_1$ at time
\begin{equation*}
\timeSch_1+L\left(\timeSch_1,\posSch_1,\{\sigma_1^{(1)},\sigma_1^{(2)}\}\right)=\frac{3\Theta+3}{\Theta(\Theta-1)}+3+\frac{3}{\Theta}=\frac{3\Theta+3}{\Theta-1}
\end{equation*}
at position $\posSch_2=-\frac{1}{\Theta}-\varepsilon$. For all $t\ge\frac{3\Theta+3}{\Theta-1}$, we have
\begin{equation*}
L(t,-\frac{1}{\Theta}-\varepsilon,\{\sigma_2\})=2+\frac{2}{\Theta}.
\end{equation*}
By assumption, we have $\Theta>\frac{1}{2}(1+\sqrt{13})$, which implies that for the finishing time $\frac{3\Theta+3}{\Theta-1}$ of schedule $S_1$ the inequality
\begin{equation*}
\frac{L\left(\frac{3\Theta+3}{\Theta-1},\posSch_2,\{\sigma_2\}\right)}{\Theta-1}=\frac{2+\frac{2}{\Theta}}{\Theta-1}<\frac{3\Theta+3}{\Theta-1}
\end{equation*}
holds. Therefore the final schedule $S_2$ is started at time $\timeSch_2=\frac{3\Theta+3}{\Theta-1}$. To sum it up, we have
\begin{equation*}
\smartstartS=\timeSch_2+L(\timeSch_2,\posSch_2,\{\sigma_2\})=\frac{3\Theta+3}{\Theta-1}+2+\frac{2}{\Theta}.
\end{equation*}
On the other hand, $\opt$  goes from the origin straight to position $-\frac{1}{\Theta}-\varepsilon$ to serve request~$\smash{\sigma_1^{(2)}}$ at time $\frac{1}{\Theta}+2\varepsilon$ (i.e., it has to wait for $\varepsilon$ units of time after it reaches position~$-\frac{1}{\Theta}-\varepsilon$) and returns to the origin at time $\frac{1}{\Theta}+3\varepsilon$. Let $q>0$ be the position of a request that has occurred by the application of \Cref{lemma: Luring} at the beginning of this proof. Then this requests is released earlier than time $q+\varepsilon$. Since $\opt$ reaches position $q$ not earlier than time~$\frac{2}{\Theta}+3\varepsilon+q>q+\varepsilon$, $\opt$ can go straight from position $-\frac{1}{\Theta}-\varepsilon$ to position $2+\frac{1}{\Theta}-\frac{\varepsilon}{2}$ serving the requests that occur by applying \Cref{lemma: Luring} as well as $\smash{\sigma_1^{(1)}}$ and $\sigma_2$ on the way. Therefore, we have
\begin{align*}
\optS&=\left|0-\left(-\frac{1}{\Theta}-\varepsilon\right)\right|+\varepsilon+\left|-\frac{1}{\Theta}-\varepsilon-\left(2+\frac{1}{\Theta}-\varepsilon\right)\right|\\
&=\frac{1}{\Theta}+3\varepsilon+2+\frac{2}{\Theta}-\frac{\varepsilon}{2}\\
&=\frac{2\Theta+3}{\Theta}+\frac{5\varepsilon}{2}.
\end{align*}
Note that $\opt$ reaches position $2+\frac{1}{\Theta}-\frac{\varepsilon}{2}$ at time $\frac{2\Theta+3}{\Theta}+\frac{5\varepsilon}{2}$ and can immediately serve~$\smash{\sigma_1^{(1)}}$ since the assumption $\Theta>\frac{1}{2}(1+\sqrt{13})$ implies 
\begin{equation*}
\frac{2\Theta+3}{\Theta}+\frac{5\varepsilon}{2}>\frac{2\Theta+3}{\Theta}+\varepsilon>\frac{3\Theta+3}{\Theta(\Theta-1)}+\varepsilon.
\end{equation*}
The latter inequality holds, because of the monotonicity of the curves $2\Theta+3$ and $3+\frac{3}{\Theta}$ and their intersection at $\Theta=\frac{1}{2}(1+\sqrt{13})$. Note furthermore that $\opt$ can serve all requests on the way even if capacity $c=1$ holds since for all requests the starting point is equal to the ending point. To sum it up, we have
\begin{equation}\label{equation: Fraction Smartstart Opt Two}
\frac{\smartstartS}{\optS}=\frac{\frac{3\Theta+3}{\Theta-1}+2+\frac{2}{\Theta}}{\frac{2\Theta+3}{\Theta}+\frac{5\varepsilon}{2}}.
\end{equation}
Note that the fraction in (\ref{equation: Fraction Smartstart Opt Two}) becomes larger with decreasing $\varepsilon$. By assumption, we have $\varepsilon<\frac{1}{25}$, which implies
\begin{equation*}
\frac{\smartstartS}{\optS}>\frac{\frac{3\Theta+3}{\Theta-1}+2+\frac{2}{\Theta}}{\frac{2\Theta+3}{\Theta}+\frac{1}{10}}=:g_2(\Theta).
\end{equation*}
The function $g_2$ is monotonically decreasing on $(\frac{1}{2}(1+\sqrt{13}),1+\sqrt{2}]$. Therefore, we have
\begin{equation*}
\frac{\smartstartS}{\optS}\ge g_2(1+\sqrt{2})=\frac{10}{573}(109+45\sqrt{2})>3>\rho^*
\end{equation*}
for all $\Theta\in(\frac{1}{2}(1+\sqrt{13}),1+\sqrt{2}]$.
\end{proof}

\begin{lemma}\label{lemma: Lower Bound Three}
Let the capacity $c\in\N\cup\{\infty\}$ of the server be arbitrary but fixed and let $1+\sqrt{2}<\Theta< 3$. There is a set of requests $\sigma$ such that
\begin{equation*}
\frac{\smartstartS}{\optS}>\rho^*\approx 2.93768.
\end{equation*}
\end{lemma}
\begin{proof}
Let $\varepsilon>0$ with $\varepsilon<\frac{1}{20}$. We apply \Cref{lemma: Luring} with $p=1$ and $\mu=\varepsilon$. For convenience, we start the enumeration of the schedules with the first schedule after the application of \Cref{lemma: Luring}. $\smartstart$ reaches position $\posSch_1=1$ at time $1+\varepsilon$. Now let the requests 
\begin{align*}
\sigma_1^{(1)}&=\left(\frac{1}{\Theta},1;\frac{1}{\Theta}+2\varepsilon\right),\\
\sigma_1^{(2)}&=\left(-\frac{1}{\Theta},-\frac{1}{\Theta};\frac{1}{\Theta}+2\varepsilon\right)
\end{align*}
appear. Note that both requests appear after time $\frac{1}{\Theta}+\varepsilon$ and therefore do not interfere with the application of \Cref{lemma: Luring}. If $\smartstart$ serves $\sigma_1^{(2)}$ before delivering $\sigma_1^{(1)}$ the time it needs is at least
\begin{equation*}
\left|1-\left(-\frac{1}{\Theta}\right)\right|+\left|\left(-\frac{1}{\Theta}\right)-1\right|=2+\frac{2}{\Theta}.
\end{equation*}
The best schedule that serves $\sigma_1^{(2)}$ after delivering $\sigma_1^{(1)}$ needs time
\begin{equation*}
\left|1-\left(\frac{1}{\Theta}\right)\right|+\left|\left(\frac{1}{\Theta}\right)-1\right|+\left|1-\left(-\frac{1}{\Theta}\right)\right|=\frac{2\Theta-2}{\Theta}+\frac{\Theta+1}{\Theta}=3-\frac{1}{\Theta}.
\end{equation*}
By assumption, we have $\Theta<3$, which implies $3-\frac{1}{\Theta}<2+\frac{2}{\Theta}$. Therefore $\smartstart$ serves~$\sigma_1^{(2)}$ after delivering $\sigma_1^{(1)}$ and for all $t\ge 1+\varepsilon$ we have
\begin{equation*}
L(t,\posSch_1,\{\sigma_1^{(1)},\sigma_1^{(2)}\})=L(t,1,\{\sigma_1^{(1)},\sigma_1^{(2)}\})=3-\frac{1}{\Theta}.
\end{equation*}
Again, by assumption, we have $\Theta<3$ and $\varepsilon<\frac{1}{20}$, which implies that for the time $1+\varepsilon$, when $\smartstart$ reaches position $\posSch_1=1$ the inequality
\begin{equation*}
\frac{L\left(1+\varepsilon,\posSch_1,\{\sigma_1^{(1)},\sigma_1^{(2)}\}\right)}{\Theta-1}=\frac{3-\frac{1}{\Theta}}{\Theta-1}>1+\varepsilon
\end{equation*}
holds. Thus, $\smartstart$ has a waiting period and starts schedule $S_1$ at time
\begin{align*}
\timeSch_1&=\min\left\{t\ge 1+\varepsilon\;\middle|\; \frac{L(t,\posSch_1,\{\sigma_1^{(1)},\sigma_1^{(2)}\})}{\Theta-1}\le t\right\}\\
&=\min\left\{t\ge 1+\varepsilon\;\middle|\; \frac{3-\frac{1}{\Theta}}{\Theta-1}\le t\right\}\\
&=\frac{3-\frac{1}{\Theta}}{\Theta-1}\\
&=\frac{3\Theta-1}{\Theta(\Theta-1)}.
\end{align*}
Next, we let the final request
\begin{equation*}
\sigma_2=\left(1,1;\frac{3\Theta-1}{\Theta(\Theta-1)}+\varepsilon\right)
\end{equation*}
appear. $\smartstart$ finishes schedule $S_1$ at time
\begin{equation*}
\timeSch_1+L(\timeSch_1,\posSch_1,\{\sigma_1^{(1)},\sigma_1^{(2)}\})=\frac{3-\frac{1}{\Theta}}{\Theta-1}+3-\frac{1}{\Theta}=\frac{3\Theta-1}{\Theta-1}.
\end{equation*}
at position $\posSch_2=-\frac{1}{\Theta}$. For all $t\ge\frac{3\Theta-1}{\Theta-1}$, we have
\begin{equation*}
L\left(t+L(\timeSch_1,-\frac{1}{\Theta},\{\sigma_2\}\right)=1+\frac{1}{\Theta}.
\end{equation*}
By assumption, we have $\Theta>1+\sqrt{2}$, which implies that for the finishing time $\frac{3\Theta-1}{\Theta-1}$ of schedule $S_1$ the inequality
\begin{equation*}
\frac{L\left(\frac{3\Theta-1}{\Theta-1},\posSch_2,\{\sigma_2\}\right)}{\Theta-1}=\frac{1+\frac{1}{\Theta}}{\Theta-1}<\frac{3\Theta-1}{\Theta-1}.
\end{equation*}
holds. Therefore the final schedule $S_2$ is started at time $\timeSch_2=\frac{3\Theta-1}{\Theta-1}$. To sum it up, we have
\begin{equation*}
\smartstartS=\timeSch_2+L(\timeSch_2,\posSch_2,\{\sigma_2\})=\frac{3\Theta-1}{\Theta-1}+1+\frac{1}{\Theta}.
\end{equation*}
On the other hand, $\opt$ goes from the origin to $-\frac{1}{\Theta}$ to collect $\smash{\sigma_1^{(1)}}$ at time $\frac{1}{\Theta}+2\varepsilon$ (i.e., it has to wait for $2\varepsilon$ units of time after reaching position $-\frac{1}{\Theta}$) and returns to the origin at time $\frac{2}{\Theta}+2\varepsilon$. Let $q>0$ be the position of a request that has occurred by the application of \Cref{lemma: Luring} at the beginning of this proof. Then this requests is released earlier than time $q+\varepsilon$. Since $\opt$ reaches position $q$ not earlier than time~$\frac{2}{\Theta}+2\varepsilon+q>q+\varepsilon$, $\opt$ can go straight from position~$-\frac{1}{\Theta}$ to position~$1$ collecting and delivering all requests that occur by the application of \Cref{lemma: Luring} as well as~$\smash{\sigma_1^{(2)}}$. Note that $\opt$ can also collect $\sigma_2$ at arrival at position $1$ at time $1+\frac{2}{\Theta}+2\varepsilon$ since the assumption $\Theta>1+\sqrt{2}$ implies
\begin{equation*}
1+\frac{2}{\Theta}+2\varepsilon=\frac{\Theta^2+\Theta-2}{\Theta(\Theta-1)}+2\varepsilon\ge\frac{3\Theta-1}{\Theta(\Theta-1)}+\varepsilon.
\end{equation*}
The latter inequality holds, because of the monotonicity of the curves $\Theta^2+\Theta-2$ and $3\Theta-1$ and intersection at $\Theta=1+\sqrt{2}$. Therefore, we have
\begin{equation*}
\optS=\left|0-\left(-\frac{1}{\Theta}\right)\right|+2\varepsilon+\left|-\frac{1}{\Theta}-1\right|=1+\frac{2}{\Theta}+2\varepsilon.
\end{equation*}
Note that $\opt$ can do this even if capacity $c=1$ holds since no additional requests need to be carried over $[0,\frac{1}{\Theta}]\cup\{1\}$, where the requests of the application of \Cref{lemma: Luring} appear. Finally, we have
\begin{equation}\label{equation: Fraction Smartstart Opt Three}
\frac{\smartstartS}{\optS}=\frac{\frac{3\Theta-1}{\Theta-1}+1+\frac{1}{\Theta}}{1+\frac{2}{\Theta}+2\varepsilon}.
\end{equation}
Note that the fraction in equality (\ref{equation: Fraction Smartstart Opt Three}) becomes larger with decreasing $\varepsilon$. By assumption, we have $\varepsilon<\frac{1}{20}$, which implies
\begin{equation*}
\frac{\smartstartS}{\optS}>\frac{\frac{3\Theta-1}{\Theta-1}+1+\frac{1}{\Theta}}{1.1+\frac{2}{\Theta}}=:g_3(\Theta).
\end{equation*}
The function $g_3$ has exactly one local minimum in the range $(1+\sqrt{2},3)$ at
\begin{equation*}
\hat{\Theta}=\frac{349}{247}+\frac{\sqrt{84998}}{247}.
\end{equation*}
Therefore, we have
\begin{equation*}
\frac{\smartstartS}{\optS}>g_3(\hat{\Theta})\approx 3.01454>3>\rho^*.
\end{equation*}
\end{proof}

\begin{lemma}\label{lemma: Lower Bound Four}
Let the capacity $c\in\N\cup\{\infty\}$ of the server be arbitrary but fixed and let $\Theta\ge 3$. There is a set of requests $\sigma$ such that 
\begin{equation*}
\frac{\smartstartS}{\optS}>\rho^*\approx 2.93768.
\end{equation*}
\end{lemma}
\begin{proof}
Let $\varepsilon>0$ with $\varepsilon<\frac{1}{75}$. We apply \Cref{lemma: Luring} with $p=1$ and $\mu=\varepsilon$. For convenience, we start the enumeration of the schedules with the first schedule after the application of \Cref{lemma: Luring}. Algorithm $\smartstart$ reaches position $\posSch_1=1$ at time $1+\varepsilon$. Now let the requests 
\begin{align*}
\sigma_1^{(1)}&=\left(\frac{\Theta+1}{2\Theta}+\frac{\varepsilon}{2},1;\frac{1}{\Theta}+2\varepsilon\right),\\
\sigma_1^{(2)}&=\left(\frac{1}{\Theta},\frac{1}{\Theta};\frac{1}{\Theta}+2\varepsilon\right)
\end{align*}
appear. Note that both requests appear after time $\frac{1}{\Theta}+\varepsilon$ and therefore do not interfere with the application of \Cref{lemma: Luring} and that the carrying path of $\smash{\sigma_1^{(1)}}$ does not cross the position~$\frac{1}{\Theta}$ of $\smash{\sigma_1^{(2)}}$ since the assumption $\Theta\ge 3$ implies $\smash{\frac{1}{\Theta}<\frac{\Theta+1}{2\Theta}}$. Thus, if $\smartstart$ serves $\smash{\sigma_1^{(2)}}$ before delivering $\smash{\sigma_1^{(1)}}$ the time it needs is at least
\begin{equation*}
\left|1-\frac{1}{\Theta}\right|+\left|\frac{1}{\Theta}-1\right|=2-\frac{2}{\Theta}.
\end{equation*}
The best schedule that serves $\sigma_1^{(2)}$ after delivering $\sigma_1^{(1)}$ needs time
\begin{align*}
&\phantom{=}\;\left|1-\left(\frac{\Theta+1}{2\Theta}+\frac{\varepsilon}{2}\right)\right|+\left|\left(\frac{\Theta+1}{2\Theta}+\frac{\varepsilon}{2}\right)-1\right|+\left|1-\frac{1}{\Theta}\right|\\
&=2\left(\frac{\Theta-1}{2\Theta}-\frac{\varepsilon}{2}\right)+\frac{\Theta-1}{\Theta}\\
&=2-\frac{2}{\Theta}-\varepsilon.
\end{align*}
Therefore $\smartstart$ serves $\sigma_1^{(2)}$ after delivering $\sigma_1^{(1)}$ and for all $t\ge 1+\varepsilon$ we have
\begin{equation*}
L\left(t,\posSch_1,\{\sigma_1^{(1)},\sigma_1^{(2)}\}\right)=L\left(t,1,\{\sigma_1^{(1)},\sigma_1^{(2)}\}\right)=2-\frac{2}{\Theta}-\varepsilon.
\end{equation*}
By assumption, we have $\Theta\ge 3$, which implies that for the finishing time $1+\varepsilon$ of schedule $S_1$ the inequality
\begin{equation*}
\frac{L\left(1+\varepsilon,\posSch_1,\{\sigma_1^{(1)},\sigma_1^{(2)}\}\right)}{\Theta-1}=\frac{2-\frac{2}{\Theta}-\varepsilon}{\Theta-1}\le1+\varepsilon
\end{equation*}
holds. Thus, the schedule $S_1$ is started immediately after the application of \Cref{lemma: Luring} at time $\timeSch_1=1+\varepsilon$. Next, we let the final request
\begin{equation*}
\sigma_2=\left(1,1;1+2\varepsilon\right)
\end{equation*}
appear. $\smartstart$ finishes schedule $S_1$ at time
\begin{equation*}
\timeSch_1+L(\timeSch_1,\posSch_1,\{\sigma_1^{(1)},\sigma_1^{(2)}\})=1+\varepsilon+2-\frac{2}{\Theta}-\varepsilon=3-\frac{2}{\Theta}
\end{equation*}
at position $\posSch_2=\frac{1}{\Theta}$. For all $t\ge 3-\frac{2}{\Theta}$, we have
\begin{equation*}
L(t,\posSch_2,\{\sigma_2\})=1-\frac{1}{\Theta}.
\end{equation*}
By assumption, we have $\Theta\ge 3$, which implies that for the finishing time $3-\frac{2}{\Theta}$ of schedule~$S_1$ the inequality
\begin{equation*}
\frac{L\left(3-\frac{2}{\Theta},\frac{1}{\Theta},\{\sigma_2\}\right)}{\Theta-1}=\frac{1-\frac{1}{\Theta}}{\Theta-1}\le 3-\frac{2}{\Theta}.
\end{equation*}
holds. Therefore the final schedule $S_2$ is started at time $\timeSch_2=3-\frac{2}{\Theta}$. To sum it up, we have
\begin{equation*}
\smartstartS=\timeSch_2+L(\timeSch_2,\posSch_2,\{\sigma_2\})=3-\frac{2}{\Theta}+1-\frac{1}{\Theta}=4-\frac{3}{\Theta}.
\end{equation*}
On the other hand, $\opt$ waits at the origin until time $2\varepsilon$. Let $q$ be the position of a request that has occurred by the application of \Cref{lemma: Luring} at the beginning of this proof. Then this requests is released earlier than time $q+\varepsilon$. Since $\opt$ reaches position $q$ not earlier than time~$q+2\varepsilon>q+\varepsilon$, $\opt$ can go straight from the origin to position $1$ collecting and delivering all requests that occur by the application of \Cref{lemma: Luring} as well as, $\smash{\sigma_1^{(1)}}$, $\smash{\sigma_1^{(2)}}$ and~$\sigma_2$. Therefore, we have
\begin{equation*}
\optS=1+2\varepsilon.
\end{equation*}
Note that $\opt$ can do this even if capacity $c=1$ holds since no additional requests need to be carried over $[0,\frac{1}{\Theta}]\cup\{1\}$, where the requests of the application of \Cref{lemma: Luring} appear. To sum it up, we have
\begin{equation}\label{equation: Fraction Smartstart Opt Four}
\frac{\smartstartS}{\optS}=\frac{4-\frac{3}{\Theta}}{1+2\varepsilon}.
\end{equation}
Note that the fraction in equality (\ref{equation: Fraction Smartstart Opt Four}) becomes larger with decreasing $\varepsilon$. By assumption, we have $\varepsilon<\frac{1}{75}$, which implies
\begin{equation*}
\frac{\smartstartS}{\optS}>\frac{4-\frac{3}{\Theta}}{1+\frac{1}{75}}=:g_4(\Theta).
\end{equation*}
The function $g_4$ is strictly monotonically increasing on $[3,\infty)$. Therefore, we have
\begin{equation*}
\frac{\smartstartS}{\optS}>g_4(3)=\frac{225}{76}>2.95>\rho^*\qedhere.
\end{equation*}
\end{proof}
We summarize the Lemmas \ref{lemma: Lower Bound One}, \ref{lemma: Lower Bound Two}, \ref{lemma: Lower Bound Three} and \ref{lemma: Lower Bound Four} into one lemma.
\begin{lemma}\label{lemma: Remaining Lower Bounds}
Let
\begin{equation*}
I_1=(1,2],\quad I_2=(\tfrac{1}{2}(1+\sqrt{13}),1+\sqrt{2}],\quad I_3=(1+\sqrt{2},3),\quad I_4=[3,\infty)
\end{equation*}
be intervals. For every $i\in\{1,2,3,4\}$ there is a set of requests $\sigma$, such that, for all $\Theta\in I_i$,
\begin{equation*}
\frac{\smartstartS}{\optS}>\rho^*\approx 2.93768.
\end{equation*}
\end{lemma}
\begin{proof}
This is an immediate consequence of the Lemmas \ref{lemma: Lower Bound One}, \ref{lemma: Lower Bound Two}, \ref{lemma: Lower Bound Three} and \ref{lemma: Lower Bound Four}.
\end{proof}

Our main theorem now follows from \Cref{theorem: General Upper Bound} combined with Propositions~\ref{proposition: Lower Bound Waiting} and~\ref{proposition: Lower Bound No Waiting}, as well as \Cref{lemma: Remaining Lower Bounds}.

\begin{theorem}\label{theorem: Main Theorem}
The competitive ratio of $\smartstart_{\Theta^*}$ is exactly
\begin{equation*}
\rho^* = f_1(\Theta^*) = f_2(\Theta^*) \approx 2.93768.
\end{equation*} 
For every other $\Theta>1$ with $\Theta\neq\Theta^*$ the competitive ratio of $\smartstart_\Theta$ is larger than $\rho^*$.
\end{theorem}
\begin{proof}
We have shown in \Cref{proposition: Lower Bound Waiting} that the upper bound 
\begin{equation*}
\frac{\smartstartS}{\optS}\le f_1(\Theta) = \frac{2\Theta^2+2\Theta}{\Theta^2+\Theta-2}
\end{equation*}
established in \Cref{proposition: Upper Bound Waiting} for the case, where $\smartstart$ waits before starting the final schedule, is tight for all $\Theta\in(2,3)$. Furthermore, we have shown in \Cref{proposition: Lower Bound No Waiting} that the upper bound
\begin{equation*}
\frac{\smartstartS}{\optS}\le f_2(\Theta) = \left(\Theta+1-\frac{\Theta-1}{3\Theta+3}\right)
\end{equation*}
established in \Cref{proposition: Upper Bound No Waiting} for the case, where $\smartstart$ does not wait before starting the final schedule, is tight for all $\Theta\in(2,\frac{1}{2}(1+\sqrt{13})]$. Since $\Theta^*\approx 2.0526$ lies in those ranges, the competitive ratio of $\smartstart_{\Theta^*}$ is indeed exactly~$\rho^*$. 

It remains to show that for every~$\Theta>1$ with~$\Theta\neq\Theta^*$ the competitive ratio is larger. First, according to \Cref{lemma: Remaining Lower Bounds}, the competitive ratio of $\smartstart$ with parameter $\Theta\in (1,2]$ or $\Theta\in (\frac{1}{2}(1+\sqrt{13}),\infty)$ is larger than $\rho^*$. By monotonicity of~$f_1$, every function value in $(2,\Theta^*)$ is larger than $f_1(\Theta^*)=\rho^*$. Thus, the competitive ratio of $\smartstart$ with parameter $\Theta\in (2,\Theta^*)$ is larger than $\rho^*$, since $f_1$ is tight on $(2,\Theta^*)$ by \Cref{proposition: Lower Bound Waiting}. Similarly, by monotonicity of~$f_2$, every function value in $(\Theta^*,\frac{1}{2}(1+\sqrt{13})]$ is larger than~$f_2(\Theta^*)=\rho^*$. Thus, the competitive ratio of $\smartstart$ with parameter~$\Theta\in (\Theta^*,\frac{1}{2}(1+\sqrt{13})]$ is larger than~$\rho^*$, since $f_1$ is tight on $(\Theta^*,\frac{1}{2}(1+\sqrt{13})]$ by \Cref{proposition: Lower Bound No Waiting}.
\end{proof}

\section{Lower Bound for the Closed Version}\label{section: Lower Bound for the Closed Version}

We provide a lower bound for \smartstart for closed online \dar on the line that matches the upper bound given in~\cite{Ascheuer1} for arbitrary metric spaces.
Note that in this setting, by definition, every schedule of \smartstart is a closed walk that returns to the origin.

\begin{theorem}\label{theorem: Competitive Ratio for Closed}
The competitive ratio of $\smartstart$ for closed online \dar on the line with $\Theta=2$ is exactly $2$. For every other $\Theta>1$ with $\Theta\neq 2$ the competitive ratio of $\smartstart_\Theta$ is larger than $2$.
\end{theorem}

\begin{proof}
We show that the competitive ratio of $\smartstart_2$ is at least~$2$ and that the competitive ratio of $\smartstart_\Theta$ is larger than $2$ for all $\Theta\neq 2$.
From the fact that \smartstart is 2-competitive even for general metric spaces~\cite[Thm.~6]{Ascheuer1}, it follows that $\smartstart_2$ has competitive ratio exactly 2 on the line.

Let $\Theta\le 2$ and consider the set of requests $\{\sigma_1\}$ with $\sigma_1=(0.5,0.5;0)$. 
Obviously, \opt can serve this request and return to the origin in time $\opt(\{\sigma_1\})=1$. 
Thus, for all $t\ge 0$, we have $L(t,0,\{\sigma_1\})=1$. 
On the other hand, \smartstart waits until time
\begin{equation*}
\timeSch_1=\frac{L(\timeSch_1,0,\{\sigma_1\})}{\Theta-1}=\frac{1}{\Theta-1}
\end{equation*}
to start its only schedule and finishes at time $\frac{\Theta}{\Theta-1}$. 
To sum it up, we have
\begin{equation*}
\frac{\smartstart(\{\sigma_1\})}{\opt(\{\sigma_1\})}=\frac{\Theta}{\Theta-1}
\end{equation*}
with~$\frac{\Theta}{\Theta-1}>2$ for all~$\Theta<2$ and~$\frac{\Theta}{\Theta-1}=2$ for~$\Theta=2$. Now let $2<\Theta\le 3$ and $\varepsilon\in (0,\min\{1-\frac{1}{\Theta-1},\frac{\Theta-2}{2(\Theta-1)}\})$, and consider the set of requests $\{\sigma_1,\sigma_2\}$ with
\begin{align*}
\sigma_1&=(0.5,0.5;0)\quad\text{and}\\
\sigma_2&=\left(1-\frac{1}{\Theta-1}-\varepsilon,1-\frac{1}{\Theta-1}-\varepsilon;\frac{1}{\Theta-1}+\varepsilon\right).
\end{align*}
By assumption, we have $\Theta>2$ and $\varepsilon <1-\frac{1}{\Theta-1}$, which implies
\begin{align*}
0 &\myoverset{$\varepsilon <1-\frac{1}{\Theta-1}$}{<}{55}1-\frac{1}{\Theta-1}-\varepsilon\\
&\myoverset{$\Theta\le 3$}{<}{55} 0.5,
\end{align*}
i.e., the position of request $\sigma_2$ lies between $0$ and $0.5$. 
If \opt moves to position~$0.5$ and then returns to the origin, it is at position
\begin{equation*}
a_2=0.5-\left|\underbrace{\left(\frac{1}{\Theta-1}+\varepsilon\right)}_{>0.5}-0.5\right|=1-\frac{1}{\Theta-1}-\varepsilon
\end{equation*}
at time $r_2=\frac{1}{\Theta-1}+\varepsilon$. 
Thus, \opt can serve $\sigma_2$ on the way and we have $\opt(\{\sigma_1,\sigma_2\})=1$. 
For all $t\ge 0$, we have $L(t,0,\{\sigma_1\})=1$. 
Therefore, \smartstart waits until time
\begin{equation*}
\timeSch_1=\frac{L(\timeSch_1,0,\{\sigma_1\})}{\Theta-1}=\frac{1}{\Theta-1}.
\end{equation*}
before starting its first schedule. Since we have $\frac{1}{\Theta-1}<\frac{1}{\Theta-1}+\varepsilon$, \smartstart starts to serve $\sigma_1$ at time~$\timeSch_1$ and returns to the origin at time $\frac{\Theta}{\Theta-1}$. 
For all $t\ge 0$, we have
\begin{equation*}
L(t,0,\{\sigma_2\})=2-\frac{2}{\Theta-1}-2\varepsilon,
\end{equation*}
thus \smartstart does not start the second and final schedule before time $\frac{2-\frac{2}{\Theta-1}-2\varepsilon}{\Theta-1}$. 
By assumption, we have $\Theta>2$, which implies~$\smash{\frac{\Theta}{\Theta-1}>\frac{2-\frac{2}{\Theta-1}-2\varepsilon}{\Theta-1}}$.
Thus, the second schedule is started at time~$\smash{\timeSch_2=\frac{\Theta}{\Theta-1}}$ and finished at time
\begin{equation*}
\smartstart(\{\sigma_1,\sigma_2\})=\frac{\Theta}{\Theta-1}+2-\frac{2}{\Theta-1}-2\varepsilon.
\end{equation*}
To sum it up, we have
\begin{align*}
\frac{\smartstart(\{\sigma_1,\sigma_2\})}{\opt(\{\sigma_1,\sigma_2\})}&\myoverset{}{=}{40}\frac{\Theta}{\Theta-1}+2-\frac{2}{\Theta-1}-2\varepsilon\\
&\myoverset{$\varepsilon <\frac{\Theta-2}{2(\Theta-1)}$}{>}{40}\frac{3\Theta-4}{\Theta-1}-2\frac{\Theta-2}{2(\Theta-1)}\\
&\myoverset{}{=}{40}2.
\end{align*}
Now let $\Theta> 3$ and $\varepsilon\in (0,0.5-\frac{1}{\Theta-1})$, and consider the set of requests $\{\sigma_1,\sigma_2\}$ with
\begin{align*}
\sigma_1&=(0.5,0.5;0)\quad\text{and}\\
\sigma_2&=\left(0.5,0.5;\frac{1}{\Theta-1}+\varepsilon\right).
\end{align*}
By assumption, we have $\varepsilon <0.5-\frac{1}{\Theta-1}$, which implies
\begin{equation*}
\frac{1}{\Theta-1}+\varepsilon<0.5,
\end{equation*}
i.e., $\sigma_2$ is released before position $0.5$ is reachable. 
If \opt moves to position~$0.5$ and then returns to the origin, it can serve both requests without additional waiting time and we have~$\opt(\{\sigma_1,\sigma_2\})=1$. 
For all $t\ge 0$, we have $L(t,0,\{\sigma_1\})=1$. 
Therefore, \smartstart waits until time
\begin{equation*}
\timeSch_1=\frac{L(\timeSch_1,0,\{\sigma_1\})}{\Theta-1}=\frac{1}{\Theta-1}.
\end{equation*}
before starting its first schedule. Since we have $\frac{1}{\Theta-1}<\frac{1}{\Theta-1}+\varepsilon$, \smartstart starts to serve $\sigma_1$ at time~$\timeSch_1$ and returns to the origin at time $\frac{\Theta}{\Theta-1}$. 
For all $t\ge 0$, we have
\begin{equation*}
L(t,0,\{\sigma_2\})=1,
\end{equation*}
thus \smartstart does not start the second and final schedule before time $\smash{\frac{1}{\Theta-1}}$. 
By assumption, we have $\Theta>3$, which implies $\smash{\frac{\Theta}{\Theta-1}>\frac{1}{\Theta-1}}$. Thus, the second schedule is started at time~$\timeSch_2=\frac{\Theta}{\Theta-1}$ and finished at time
\begin{equation*}
\smartstart(\{\sigma_1,\sigma_2\})=\frac{\Theta}{\Theta-1}+1.
\end{equation*}
To sum it up, we have
\begin{equation*}
\frac{\smartstart(\{\sigma_1,\sigma_2\})}{\opt(\{\sigma_1,\sigma_2\})}=\frac{\Theta}{\Theta-1}+1>2.
\end{equation*}
\end{proof}
\newpage
\bibliography{full}
\addcontentsline{toc}{chapter}{Bibliography}

\newpage
\appendix

\section{Algorithm {\normalfont\ignore}}\label{appendix: Algorithm Ignore}

\setcounter{theorem}{0}
\setcounter{section}{1}
\renewcommand{\thetheorem}{\Alph{section}.\arabic{theorem}}

The algorithm $\ignore$ was described in \cite{Ascheuer1} (though the authors do not claim originality for the algorithm) for the closed case of the online \dar problem in arbitrary metric spaces. 
We describe the algorithm for the open case as it was introduced in \cite{Krumke1} (see \Cref{algorithm: Ignore}): 
The server remains idle until the point in time~$t$ when the first request appears. 
It then serves the requests released at time~$t$ immediately by following a shortest schedule~$S$. 
All requests that appear during the time when the algorithm follows $S$ are temporarily ignored. After $S$ has been completed the server is at the destination of the last served request~$p$, computes a shortest schedule for the unserved requests starting in position~$p$ and follows this schedule. Again all new requests appearing during the time that the server is following the schedule, are temporarily ignored. The algorithm keeps on following schedules and temporarily ignoring requests this way.

\begin{algorithm}
\SetKwBlock{Repeat}{repeat}{}
\DontPrintSemicolon
\Repeat{
 \eIf{$R_t\neq\emptyset$}{
 Start optimal offline schedule serving $R_t$ starting from the current position\;
   }
   {wait\;}}
 \caption{$\ignore$}\label{algorithm: Ignore}
\end{algorithm}

It was shown in \cite{Krumke1} that $4$ is an upper bound for the competitive ratio of $\ignore$. 
We show that this is tight on the line.

\begin{proposition} 
The competitive ratio of $\ignore$ is $\rho_\ignore=4$.
\end{proposition}
\begin{proof}
It was shown in \cite[Theorem 2.29]{Krumke1} that $4$ is a upper bound for the competitive ratio of $\ignore$ for arbitrary metric spaces and therefore in particular for the real line. It remains to show that for every $\varepsilon>0$ there is a set of requests $\sigma$ such that
\begin{equation*}
\frac{\ignoreS}{\optS} \geq 4-\varepsilon.
\end{equation*}
Let $\varepsilon>0$. 
We consider the set of requests $\sigma$ consisting of
\begin{align*}
\sigma_1&=\left(1-\frac{1}{5}\varepsilon,1-\frac{1}{5}\varepsilon;0\right),\\
\sigma_2^{(1)}&=\left(\frac{1}{2},1-\frac{1}{5}\varepsilon;\frac{1}{5}\varepsilon\right),\\
\sigma_2^{(2)}&=\left(0,0;\frac{1}{5}\varepsilon\right),\\
\sigma_3&=\left(1-\frac{1}{5}\varepsilon,1-\frac{1}{5}\varepsilon;1\right).
\end{align*}
$\ignore$ first serves request $\sigma_1$ in time $1-\frac{1}{5}\varepsilon$. 
Then, it serves the requests~$\sigma_2^{(1)}$ and~$\sigma_2^{(2)}$. 
Note that serving $\sigma_2^{(1)}$ before collecting $\sigma_2^{(2)}$ takes time $2-\frac{3}{5}\varepsilon$, while collecting~$\sigma_2^{(2)}$ first takes time $2-\frac{2}{5}\varepsilon$. 
Therefore $\ignore$ serves $\sigma_2^{(1)}$ first and the second schedule ends in the origin. 
The third and final schedule again needs time~$1-\frac{1}{5}$ to serve $\sigma_3$. 
To sum it up, we have
\begin{equation*}
\ignoreS=4-\varepsilon.
\end{equation*}
$\opt$ on the other hand waits until time~$\frac{1}{5}\varepsilon$ at the origin for the request $\sigma_2^{(2)}$ and then just collects and delivers the remaining requests on its way to $p=1-\frac{1}{5}\varepsilon$, resulting in
\begin{equation*}
\optS=1.
\end{equation*}
Therefore, we have
\begin{equation*}
\rho_{\text{$\ignore$}}=\frac{\ignoreS}{\optS}=4-\varepsilon,
\end{equation*}
as desired.
\end{proof}
\end{document}